\documentclass[reqno,12pt]{amsart}%
\usepackage[body={17cm,22cm}]{geometry}

\usepackage[T1]{fontenc}
\usepackage{graphicx}
\usepackage{mathrsfs}
\usepackage[french]{babel}
\usepackage{amscd,latexsym,amsthm,amsfonts,amssymb,amsmath,amsxtra}
\usepackage[colorlinks, urlcolor=blue,  citecolor=blue]{hyperref}
\usepackage[all]{xy}%
\providecommand{\U}[1]{\protect\rule{.1in}{.1in}}
\RequirePackage{amsmath}
\RequirePackage{amssymb}

\newcommand{\BC}{{\mathbb {C}}}

\newcommand{\BG}{{\mathbb {G}}}

\newcommand{\BQ}{{\mathbb {Q}}}

\newcommand{\BZ}{{\mathbb {Z}}}

\newcommand{\CB}{{\mathcal {B}}}
\newcommand{\CC}{{\mathcal {C}}}

\newcommand{\CH}{{\mathcal {H}}}

\newcommand{\CO}{{\mathcal {O}}}

\newcommand{\CS}{{\mathcal {S}}}
\newcommand{\CT}{{\mathcal {T}}}

\newcommand{\RA}{{\mathbf {A}}}

\newcommand{\Aut}{{\mathrm{Aut}}}

\newcommand{\Br}{{\mathrm{Br}}}

\newcommand{\coker}{{\mathrm{coker}}}

\newcommand{\Gal}{{\mathrm{Gal}}}

\newcommand{\Hom}{{\mathrm{Hom}}}

\renewcommand{\Im}{{\mathrm{Im}}}

\newcommand{\Ker}{{\mathrm{Ker}}}

\newcommand{\Mor}{{\mathrm{Mor}}}

\newcommand{\Pic}{\mathrm{Pic}}

\newcommand{\Spec}{{\mathrm{Spec}}}

\font\cyr=wncyr10
\newcommand{\Sha}{\hbox{\cyr X}}

\newcommand{\iso}{\stackrel{\sim}{\rightarrow} }

\newcommand{\sbt}{\subset}

\newcommand{\bk}{\bar{k}}

\newcommand{\et}{{\rm{\acute et}}}

\numberwithin{equation}{section}

\theoremstyle{remark}

\newtheorem{defi}{\rm{\textbf{D\'efinition}}}[section]

\newtheorem{exam}[defi]{\rm{\textbf{Exemple}}}

\newtheorem{rem}[defi]{\rm{\textbf{Remarque}}}

\theoremstyle{plain}

\newtheorem{thm}[defi]{\rm{\textbf{Th\'eor\`eme}}}

\newtheorem{cor}[defi]{\rm{\textbf{\textbf{Corollaire}}}}

\newtheorem{lem}[defi]{\rm{\textbf{Lemme}}}

\newtheorem{prop}[defi]{\rm{\textbf{\textbf{Proposition}}}}

\begin{document}

\title[]
{Sous-groupe de Brauer invariant et obstruction de descente it\'er\'ee}

\author{Yang CAO}

\address{Yang CAO \newline Leibniz Universit\"at Hannover
\newline Welfengarten 1, 30167 Hannover, Allemagne}

\email{yang.cao@math.uni-hannover.de; yangcao1988@gmail.com}

\date{\today.}

\maketitle

\begin{abstract}
Pour une vari\'et\'e quasi-projective, lisse, g\'eom\'e\-triquement int\`egre sur un corps de nombres $k$,
on montre que l'obstruction de descente it\'er\'ee est \'equivalente \`a  l'obstruction de descente.
Ceci g\'en\'eralise un r\'esultat de Skorobogatov, et ceci r\'epond \`a une question ouverte de Poonen.
Les outils principaux sont
 la notion de sous-groupe de Brauer invariant et la notion d'obstruction de Brauer-Manin \'etale invariante
pour une $k$-vari\'et\'e munie d'une action d'un groupe lin\'eaire connexe.
 
  \medskip

\indent
Summary. 
For a quasi-projective smooth geometrically integral variety over a number field $k$,
we prove that the iterated descent obstruction is equivalent to the descent obstruction.
This generalizes a result of Skorobogatov, and this answers an open question of Poonen.
Our main tools are
the notion of invariant Brauer subgroup and the notion of invariant \'etale Brauer-Manin obstruction
for a $k$-variety equipped with an action of a connected linear algebraic group.
\end{abstract}

\tableofcontents

\section{Introduction}

Soit $k$ un corps de nombres.
Soit $ \RA_k $ l'anneau des ad\`eles de $k$.
Pour une $k$-vari\'et\'e lisse $X$, 
on note $X(\RA_k)$ l'ensemble des points ad\'eliques de $X$.
On a le plongement diagonal
$$ X(k) \subset X(\RA_{k}).$$
C'est une question importante de caract\'eriser l'adh\'erence des points rationnels dans les points ad\'eliques
(principe de Hasse, approximation faible, approximation forte).
Manin en 1970 a montr\'e que cette adh\'erence est contenue dans un ferm\'e d\'etermin\'e par le groupe de Brauer de la vari\'et\'e $X$  (\cite{Ma}).  
Depuis lors, divers auteurs (Manin, Colliot-Th\'el\`ene, Sansuc, Skorobogatov, Harari, Demarche)  ont d\'ecrit d'autres ferm\'es de $ X(\RA_{k})$ contenant les points rationnels, 
et se sont attach\'es \`a comprendre les inclusions entre ces divers ferm\'es. 
On a utilis\'e pour cela  les torseurs sous des groupes lin\'eaires (finis ou non) sur $X$, et on a utilis\'e des combinaisons de ces
deux approches pour d\'eterminer des ferm\'es minimaux de $X(\RA_{k})$ contenant $X(k)$.
Harari et Skorobogatov ont d\'ecrit une inclusion  (\cite[D\'ef. 4.2]{HS02}, cf. (\ref{BiDeqdefdesc}) pour la d\'efinition)
$$X(k) \subset X(\RA_{k})^{\rm descent}.$$
Ensuite Poonen a it\'er\'e cette inclusion en (\cite[\S 8.5.2]{P1}, cf. (\ref{BiDeqdefdescdesc}) pour la d\'efinition)
 $$X(k) \subset X(\RA_{k})^{\rm descent, descent}  \subset X(\RA_{k})^{\rm descent},$$
 et demand\'e (cf. \cite[\S 8.5.4]{P1}) si la deuxi\`eme inclusion raffine la premi\`ere.
 Le th\'eor\`eme principal du pr\'esent article (th\'eor\`eme \ref{BiDthm2}) permet de r\'epondre
 \`a cette question de Poonen : $X(\RA_{k})^{\rm descent, descent}  = X(\RA_{k})^{\rm descent}$
(th\'eor\`eme \ref{BiDcor2} ci-dessous).
 Ce th\'eor\`eme \ref{BiDcor2}  apporte un point final \`a l'utilisation combin\'ee du groupe de Brauer
et de la descente sous des groupes lin\'eaires dans la d\'etermination de l'adh\'erence
de $X(k) $ dans $X(\RA_{k}).$

\medskip

Donnons maintenant des \'enonc\'es pr\'ecis.

 On note $\Omega_k $ l'ensemble des places du corps de nombres $ k $.
Pour chaque $ v \in \Omega_k $, on note $k_v$ le compl\'et\'e de $k$ en $v$ et  $ \CO_v\sbt k_v$ l'anneau des entiers ($\CO_v=k_v$ pour $v$ archim\'edienne).  

Pour $B$ un sous-groupe de $\Br (X)$, on d\'efinit
$$ X ({\RA}_k)^ B = \{(x_v)_{v \in \Omega_k} \in X ({\RA}_k): \ \ \sum_{v \in \Omega_k} \ inv_v (\xi (x_v)) = 0\in \BQ/\BZ, \ \ \forall \xi \in B \}.$$ 
Comme l'a remarqu\'e Manin (\cite{Ma}), la th\'eorie du corps de classes donne $ X (k) \subseteq X (\RA_k)^B $.

Soient $F$ un $k$-groupe alg\'ebrique et  $f: Y\to X$ un $F$-torseur. 
Pour tout 1-cocycle $\sigma\in Z^1(k,F)$, on note $F_{\sigma}$, respectivement  $f_{\sigma}: Y_{\sigma}\to X$ le   tordu du $k$-groupe 
$F$, respectivement du torseur $f$,
 par le 1-cocycle $\sigma$.
Alors $f_{\sigma}$ est un $F_{\sigma}$-torseur. La classe d'isomorphisme du $k$-groupe $F_{\sigma}$, respectivement du torseur $f_{\sigma}$, ne d\'epend que de la classe de $\sigma$ dans $H^1(k,F)$.
Par abus de notation, \'etant donn\'ee une classe $[\sigma] \in H^1(k,F)$, on
note $F_{\sigma}=F_{[\sigma]}$ et $f_{\sigma}=f_{[\sigma]}$.

Pour une $k$-vari\'et\'e lisse $X$, Skorobogatov (\cite{Sk99}) et Poonen d\'efinissent (\cite[\S 3.3]{P}) l'ensemble suivant
\begin{equation}\label{BiDthm2e}
X(\RA_k)^{\et, \Br}:=\bigcap_{\stackrel{f: Y\xrightarrow{F}X,}{ F\ \text{fini}}} \bigcup_{\sigma\in H^1(k,F)}f_{\sigma}(Y_{\sigma}(\RA_k)^{\Br(Y_{\sigma})}) ,
\end{equation}
  o\`u $F$ parcourt les $k$-groupes finis.
Ils obtiennent une inclusion $X(k)\sbt X(\RA_k)^{\et, \Br}$.
Ceci d\'efinit une obstruction au principe de Hasse pour $X$, appel\'ee \emph{obstruction de Brauer-Manin \'etale},
\'etudi\'ee dans le cas projectif par Skorobogatov, Harari et Demarche, puis dans le cas quasi-projectif \cite{CDX}.

Le r\'esultat principal de cet article est :

\begin{thm}\label{BiDthm2}
Soient $G$ un $k$-groupe lin\'eaire quelconque, $Z$ une $k$-vari\'et\'e lisse et $p: X\to Z$ un $G$-torseur.
Alors: 
$$Z(\RA_k)^{\et, \Br}=\cup_{\sigma\in H^1(k,G)}p_{\sigma}(X_{\sigma}(\RA_k)^{\et,\Br}).$$
\end{thm}

Pour $G$  fini et $Z$ projective,
ce th\'eor\`eme avait d\'ej\`a \'et\'e \'etabli par Skorobogatov (\cite[Thm. 1.1]{Sk1}).
 Pour $G$ fini et $Z$ quasi-projective,
il avait \'et\'e  ensuite \'etabli par Demarche, Xu et l'auteur (\cite[Prop. 6.6]{CDX}).
Si $Z$ est projective,  $\pi_1(Z_{\bk})$ est fini et  $G$ est une extension d'un $k$-groupe fini par un tore,
Balestrieri a \'etabli une variante simple dans \cite[Thm. 1.9]{Bale},
o\`u elle consid\`ere l'obstruction de Brauer-Manin alg\'ebrique \'etale.

Par ailleurs, dans \cite[\S 3.2]{P} et \cite[\S 8]{P1}, on d\'efinit deux ensembles 
\begin{equation}\label{BiDeqdefdesc}
X(\RA_k)^{\text{descent}}:=\bigcap_{\stackrel{f: Y\xrightarrow{F}X,}{ F\ \text{lin\'eaire}}} \bigcup_{\sigma\in H^1(k,F)}f_{\sigma}(Y_{\sigma}(\RA_k)),
\end{equation}
\begin{equation}\label{BiDeqdefdescdesc}
X(\RA_k)^{\text{descent,descent}}:=\bigcap_{\stackrel{f: Y\xrightarrow{F}X,}{ F\ \text{lin\'eaire}}} \bigcup_{\sigma\in H^1(k,F)}f_{\sigma}(Y_{\sigma}(\RA_k)^{\text{descent}}).
\end{equation}
On a $X(k) \sbt X(\RA_k)^{\text{descent}}$ 
et $X(k)  \sbt X(\RA_k)^{\text{descent,descent}}$.
 Ceci d\'efinit deux nouvelles obstructions au principe de Hasse pour $X$, appel\'ees
\emph{obstruction de descente} et \emph{obstruction de descente it\'er\'ee}.
D'apr\`es la s\'erie de travaux (\cite{D09}, \cite{Sk1} et \cite{CDX}), 
on a $X(\RA_k)^{\et, \Br}=X(\RA_k)^{\text{descent}}$ lorsque $X$ est quasi-projective (\cite[Thm. 1.5]{CDX}).
Du th\'eor\`eme \ref{BiDthm2} on d\'eduit facilement le :
\begin{thm}\label{BiDcor2}
Pour toute vari\'et\'e quasi-projective lisse g\'eom\'e\-triquement int\`egre $X$, 
on a $X(\RA_k)^{\text{descent,descent}}=X(\RA_k)^{\text{descent}}$.
\end{thm}

\bigskip

L'id\'ee cl\'e de la d\'emonstration du th\'eor\`eme \ref{BiDthm2} est la notion de  sous-groupe de Brauer invariant (\cite[D\'ef. 3.1]{C1}), que nous rappelons ici :

\begin{defi}\label{BiDdefBrainv}
Soit $G$ un  groupe alg\'ebrique connexe.

(1) Soit $(X,\rho )$ une $G$-vari\'et\'e lisse connexe.  
\emph{Le sous-groupe de Brauer $G$-invariant} de $X$ est le sous-groupe 
$$\Br_G(X):=\{b\in \Br(X)\ :\ (\rho^*(b)-p_2^*(b))\in p_1^*\Br(G)\}$$
de $\Br(X)$, o\`u $G\times X\xrightarrow{p_1}G$, $G\times X\xrightarrow{p_2}X$ sont les projections 
et $G\times X\xrightarrow{\rho}X$ est l'action de $G$.

(2) Soit $X$ une $G$-vari\'et\'e lisse quelconque.
\emph{Le sous-groupe de Brauer $G$-invariant} de $X$ est le sous-groupe 
$\Br_G(X)\sbt \Br(X)$ des \'el\'ements $\alpha$
 v\'erifiant $\alpha|_{X'}\in \Br_G(X')$ pour toute composante connexe $X'$ de $X$.

(3) Soient $F$ un $k$-groupe fini et  $X$ une $G$-vari\'et\'e lisse quelconque.  
Un $F$-torseur $Y\xrightarrow{f}X$ est \emph{$G$-compatible} s'il existe une action de $G$ sur $Y$ telle que $f$ soit un $G$-morphisme.
\end{defi}

D'apr\`es la proposition \ref{BiDprop2.2}, l'action de $G$ sur $Y$ v\'erifiant les conditions ci-dessus est unique et
  le $F_{\sigma}$-torseur $f_{\sigma}$ est aussi $G$-compatible pour tout $\sigma\in H^1(k,F)$. 
On d\'efinit la variante de $X(\RA_k)^{\et,\Br}$ suivante:
\begin{equation}\label{BiDdef1e1}
X(\RA_k)^{G-\et, \Br_G}:=\bigcap_{\stackrel{f: Y\xrightarrow{F}X\ G-\text{compatible} ,}{ F\ \text{fini}}} \bigcup_{\sigma\in H^1(k,F)}f_{\sigma}(Y_{\sigma}(\RA_k)^{\Br_G(Y_{\sigma})}) .
\end{equation}
Alors $X(k)\sbt X(\RA_k)^{\et,\Br}\sbt X(\RA_k)^{G-\et, \Br_G}$.
Ceci d\'efinit une obstruction au principe de Hasse pour $X$, appel\'ee \emph{obstruction de Brauer-Manin \'etale invariante}.

Le th\'eor\`eme suivant joue un r\^ole cl\'e dans la d\'emonstration du th\'eor\`eme \ref{BiDthm2}.

\begin{thm}\label{BiDthm1}
Soient $G$ un groupe lin\'eaire connexe et $X$ une $G$-vari\'et\'e lisse. 
Alors 
$$X(\RA_k)^{G-\et, \Br_G}=X(\RA_k)^{\et, \Br} .$$
\end{thm}

Dans le cas   o\`u $X$ est un $G$-espace homog\`ene \`a stabilisateur g\'eom\'etrique connexe,
d'apr\`es le corollaire \ref{BiDrem2.2} (4), tout torseur  $G$-compatible  sous un $k$-groupe fini est constant.
Donc on peut obtenir facilement le r\'esultat suivant.

\begin{cor}\label{BiDcor1}
Soient $G$ un groupe lin\'eaire connexe et $X$ un $G$-espace homog\`ene \`a stabilisateur g\'eom\'etrique connexe. Alors
$$X(\RA_k)^{\et, \Br} =X(\RA_k)^{G-\et, \Br_G}=X(\RA_k)^{\Br_G(X)}.$$
\end{cor}

Ce r\'esultat  particulier peut s'\'etablir aussi via l'approximation forte sur $X$ par rapport \`a $\Br_G(X)$ (voir \cite[Thm. 1.4]{BD}).

\bigskip

Donnons maintenant la structure de l'article. 

 Au \S 2, sur un corps $k$ quelconque, s'inspirant de la notion de torseur universel de Colliot-Th\'el\`ene et Sansuc, 
 on introduit la notion de torseur universel de $n$-torsion (D\'ef. \ref{BiDdef2.20}).
  Ensuite,   on   utilise cette notion \`a \'etablir une formule de K\"unneth sp\'eciale pour la cohomologie \'etale de degr\'e $2$.

Au \S 3, sur un corps $k$ de caract\'eristique z\'ero,
on consid\`ere la donn\'ee d'un $k$-groupe alg\'ebrique $G$, d'une $G$-vari\'et\'e $X$ lisse,
 d'un $k$-groupe fini  $F$, d'un torseur $Y \to X$ sous $F$,  on donne des conditions \'equivalentes
 pour le rel\`evement, de fa\c con compatible, de l'action de $G$ sur $X$ en une action sur $Y$.
 Ce rel\`evement n'est pas toujours possible. 
 On \'etudie les homomorphismes surjectifs de groupes  alg\'ebriques connexes  $H\to G$
 avec une action compatible de $H$ sur $Y$, et on montre qu'il existe un objet minimal $H_Y$.
 \'Etant donn\'e un \'el\'ement $\alpha \in \Br(X)$, en utilisant la formule de K\"unneth ci-dessus,
 on montre ensuite qu'il existe un torseur $Y \to X$
 sous un $k$-groupe fini commutatif $F$ tel que l'image r\'eciproque de $\alpha$
dans $\Br(Y)$ soit invariante sous $H_Y $.

Au \S 4, on rappelle des notions et des r\'esultats \'etablis dans \cite[\S 3]{C1}, 
en particulier, la notion de sous-groupe de Brauer invariant et aussi ses propri\'et\'es \'el\'ementaires.
Ces r\'esultats seront utilis\'es  dans les \S 5 et \S 6. 

Au \S 5, le corps de base $k$ est un corps de nombres.
Dans \cite{C1}, \'etant donn\'e un torseur $Y \to X$ sous un groupe lin\'eaire connexe $G$,
j'ai d\'evelopp\'e la m\'ethode de descente des points ad\'eliques orthogonaux
aux sous-groupes de Brauer invariants.  Au \S 5, on donne deux nouvelles variantes
de cette descente. La premi\`ere (Proposition \ref{BiDprop4.1}) traite du cas o\`u $G$ est un $k$-groupe fini commutatif.
La seconde  (Proposition \ref{BiDprop4.2}) implique l'obstruction de Brauer-Manin \'etale invariante.

Les \S 3, \S 4 et \S 5  sont utilis\'es de fa\c con essentielle au  \S 6 o\`u l'on \'etablit le  th\'eor\`eme   \ref{BiDthm1}.

Au \S 7,  en combinant le th\'eor\`eme  \ref{BiDthm1} et la proposition \ref{BiDprop4.2},
on \'etablit  les th\'eor\`emes  \ref{BiDthm2} et  \ref{BiDcor2}.

 \bigskip

\textbf{Conventions et notations}. 

Soit $k$ un corps quelconque de caract\'eristique $\mathrm{char}(k)$. On note $\overline{k}$ une cl\^oture alg\'ebrique, $k_s$ une cl\^oture s\'eparable et $\Gamma_k:=\Gal(k_s/k)$.
Si $\mathrm{char}(k)=0$, on a $k_s=\bk$ et $\Gamma_k:=\Gal(\bk/k)$.

Tous les groupes de cohomologie sont des groupes de cohomologie \'etale.

Une $k$-vari\'et\'e $X$ est un $k$-sch\'ema s\'epar\'e de type fini. 
 Pour $X$ une telle vari\'et\'e, on note $k[X]$ son anneau des fonctions globales,
$k[X]^{\times}$ son groupe des fonctions inversibles,
$\Pic(X):=H^1_{\text{\'et}}(X,\BG_m)$ son groupe de Picard et
$\Br(X):=H_{\text {\'et}}^2 (X, \BG_m)$ son groupe de Brauer. Notons
$$\Br_1 (X) := \Ker [\Br (X) \to\Br (X_ {\bk})]\ \ \text{ et}\ \ \Br_a(X):=\Br_1(X)/\Im\Br(k).$$
Le groupe $\Br_1 (X)$ est le sous-groupe ``alg\'ebrique'' du groupe de Brauer de $X$.
Si $X$ est int\`egre, on note $k(X)$ son  corps des fonctions rationnelles 
et $\pi_1(X,\bar{x})$ (ou $\pi_1(X)$) son groupe fondamental \'etale, o\`u $\bar{x}$ est un point g\'eom\'etrique de $X$.
Soit $\pi_1(X_{k_s})^{ab}$ le quotient maximal ab\'elien de $\pi_1(X_{k_s})$.
Alors $\pi_1(X_{k_s})^{ab} $ est un $\Gamma_k$-module.

Un $k$-groupe alg\'ebrique $G$ est une $k$-vari\'et\'e qui est un $k$-sch\'ema en groupes. 
On note $e_G$ l'unit\'e de $G$ et $G^*:=\Hom_{k_s-\text{groupe}}(G_{k_s},\BG_m)$ le groupe des caract\`eres de $G_{k_s}$.
C'est un module galoisien de type fini.
De plus, si $G$ est connexe sur $\BC$, le groupe $\pi_1(G)$ est commutatif (cf. \cite[Prop. 1.1 (2)]{BS}).

Un $k$-groupe fini $F$ est un $k$-groupe alg\'ebrique qui est fini sur $k$. 
Dans ce cas, $F$ est d\'etermin\'e par le $\Gamma_k$-groupe $F(k_s)$.
Pour toute $k$-vari\'et\'e lisse connexe $X$, 
on a un isomorphisme canonique (\cite[\S XI.5]{SGA1}):
\begin{equation}\label{BiDprop2.2e}
H^1(\pi_1(X),F(k_s))\iso H^1(X,F)\ \ \ \text{et donc}\ \ \ H^1(X_{k_s},F)\cong \Hom_{cont}(\pi_1(X_{k_s}),F(k_s))/\sim 
\end{equation}
o\`u l'action de $\pi_1(X)$ sur $F(k_s)$ est induite par celle de $\Gamma_k$ et $\sim $ est induite par la conjugaison.

 Soit $G$ un $k$-groupe alg\'ebrique. Une \emph{$G$-vari\'et\'e} $(X,\rho )$ (ou $X$) est une $k$-vari\'et\'e $X$ munie d'une action \`a gauche $G\times_k X\xrightarrow{\rho}X$. 
Un $k$-morphisme de $G$-vari\'et\'es est appel\'e \emph{$G$-morphisme} s'il est compatible avec l'action de $G$.

Comme d\'ej\`a indiqu\'e ci-dessus, pour tout $k$-groupe alg\'ebrique $F$, tout $F$-torseur $f: Y\to X$ et tout $\sigma\in H^1(k,F)$,
 on note $F_{\sigma}$ (resp. $f_{\sigma}: Y_{\sigma}\to X$) le tordu de $F$ (resp. de $f$). Ainsi $f_{\sigma}$ est un $F_{\sigma}$-torseur.

\section{Torseur universel de n-torsion et formule de K\"unneth de degr\'e 2}

Dans toute cette section,  $k$ est un corps quelconque. Sauf  mention explicite du contraire,  une vari\'et\'e est une $k$-vari\'et\'e.
Fixons un entier $n\geq 2$ avec $\mathrm{char}(k)\nmid n$ et notons $-\otimes -:=-\otimes_{\BZ/n}-$.

Cette section contient deux parties.
 On introduit d'abord la version de $n$-torsion de la notion de torseur universel (Colliot-Th\'el\`ene et Sansuc) dans la d\'efinition \ref{BiDdef2.20},
 et aussi la notion de type prolong\'e d'un torseur (Harari et Skorobogatov) dans la proposition \ref{BiDprop2.1}.
En utilisant ces notions, on consid\`ere ensuite le cup-produit de la cohomologie \'etale de degr\'e 2 sur un produit de deux vari\'et\'es quelconques  et
 on \'etablit une formule de K\"unneth pour ce produit (Prop. \ref{BiDlem2.3.1}).
Cette formule g\'en\'eralise un r\'esultat de Skorobogatov et Zarhin, qui traite du cas o\`u les deux vari\'et\'es sont propres.

\medskip

Soient $\mathrm{Sh}(k)$ la cat\'egorie des faisceaux \'etales sur le petit site de $\Spec\ k$ et $ D^+(k)$ la cat\'egorie d\'eriv\'ee  born\'ee \`a gauche de $\mathrm{Sh}(k)$
et $D^b(k)$ la cat\'egorie d\'eriv\'ee  born\'ee de $\mathrm{Sh}(k)$ (une sous-cat\'egorie pleine de $D^+(k)$).
Pour tout $i\in \BZ$, on a les sous-cat\'egories canoniques $D^{\geq i}(k)$ et $D^{\leq i}(k)$ de $ D^+(k)$ et deux foncteurs canoniques $\tau_{\leq i}$, $\tau_{\geq i}$ (\cite[D\'ef. 12.3.1 et Prop. 13.1.5]{KS}).
Donc $\mathrm{Sh}(k)=D^{\geq 0}(k)\cap D^{\leq 0}(k)$ est une sous-cat\'egorie pleine canonique de $D^+(k)$.
Par abus de notation, pour un objet $M$ de $\mathrm{Sh}(k)$, on note $M$ l'objet de $D^+(k)$ repr\'esent\'e par le complexe qui consiste en $M$ en degr\'e $0$.

Soient $X$ une vari\'et\'e g\'eom\'etriquement int\`egre et $p: X\to \Spec\ k$ le morphisme de structure.
Soit $S_X$ un groupe de type multiplicatif tel que $S_X^*\cong H^1(X_{k_s},\mu_n)$ comme $\Gamma_k$-modules.
On rappelle que $H^1(X_{k_s},\mu_n)$ est fini.

Dans $D^+(k)$, il existe deux morphismes canoniques $\BG_m\to Rp_*\BG_m$ et $\mu_n\to Rp_*\mu_n$.
Soient $\Delta$ le c\^one de $\BG_m[1]\to Rp_*\BG_m[1]$ et $\Delta_n$ le c\^one de $\mu_n[1]\to Rp_*\mu_n[1]$.
La suite exacte de Kummer donne un diagramme commutatif de triangles distingu\'es:

$$\xymatrix{\Delta_n[-2]\ar[r]\ar[d]^{\psi}&\mu_n\ar[r]\ar[d]&Rp_*\mu_n\ar[r]^{+1}\ar[d]&\\
\Delta[-2]\ar[r]\ar[d]^{n\cdot }&\BG_m\ar[r]\ar[d]^{n\cdot }&Rp_*\BG_m\ar[r]^{+1}\ar[d]^{n\cdot }&\\
\Delta[-2]\ar[r]\ar[d]^{+1}&\BG_m\ar[r]\ar[d]^{+1}&Rp_*\BG_m\ar[r]^{+1}\ar[d]^{+1}&\\
&&&.
}$$
Les faisceaux de cohomologie des complexes $\Delta_n$ et $\Delta$ se calculent comme suit:
$$\Delta_n\in D^{\geq 0}(k),\ \ \ \CH^0(\Delta_n)\cong H^1(X_{k_s},\mu_n) \cong S_X^*, $$ 
$$\Delta\in D^{\geq -1}(k),\ \ \ \CH^{-1}(\Delta)=k_s[X]^{\times}/k_s^{\times}\ \ \ \text{et}\ \ \  \CH^0(\Delta)=\Pic(X_{k_s}). $$
Le morphisme $\psi: \Delta_n\to \Delta $ induit un morphisme $\psi_{\leq 0}:=\tau_{\leq 0}\psi: S_X^*\to \tau_{\leq 0}\Delta$.

Harari et Skorobogatov montrent que, pour tout groupe de type multiplicatif $S$, on a une suite exacte naturelle (\cite[Prop. 1.1]{HS13}, o\`u $\tau_{\leq 0}\Delta$ est not\'e   $KD'(X)$):
\begin{equation}\label{BiD2e1}
H^1(k,S)\to H^1(X,S)\xrightarrow{\chi} \Hom_{D^+(k)}(S^*,\tau_{\leq 0}\Delta)\to H^2(k,S).
\end{equation} 

\begin{defi}\label{BiDdef2.20}
Un \emph{torseur universel de $n$-torsion} pour $X$ est un $S_X$-torseur $\CT_X$ sur $X$
 tel que $\chi([\CT_X])=\psi_{\leq 0}: S_X^*\to \tau_{\leq 0}\Delta$.
\end{defi}

D'apr\`es \cite[Prop. 1.3]{HS13}, si $X(k)\neq \emptyset$, pour chaque $x\in X(k)$,  
il existe alors un unique torseur universel de $n$-torsion $\CT_X$ pour $X$ tel que $x^*[\CT_X]=0\in H^1(k,S_X)$.

Dans le cas o\`u $k$ est un corps de nombres,  il existe un torseur universel de $n$-torsion pour $X$ lorsque $X(\RA_k)^{\Br_1(X)}\neq \emptyset$ (\cite[Cor. 3.6]{HS13}).

\begin{prop}\label{BiDprop2.1}
Soit $\CT_X$ un torseur universel de $n$-torsion pour $X$.
Soit $S$ un groupe de type multiplicatif tel que $n\cdot S=0$.
Alors, pour tout $S$-torseur $Y$ sur $X$, il existe un unique homomorphisme $\phi: S_X\to S $ tel que
$$\phi_*([\CT_X])-[Y]\in \Im (H^1(k,S)\to H^1(X,S)). $$
\end{prop}

\begin{proof}
Le triangle $\Delta_n\to \Delta \stackrel{n \cdot }{\to} \Delta \stackrel{+1}{\to}$ induit une suite exacte
$$\Hom_{D^+(k)}(S^*,\Delta[-1])\to \Hom_{D^+(k)}(S^*,\Delta_n)\to \Hom_{D^+(k)}(S^*,\Delta) \xrightarrow{n\cdot }\Hom_{D^+(k)}(S^*,\Delta).$$
Puisque $S^*\in D^{\leq 0}(k)$ et $n\cdot S^*=0$, on a
$$\Hom_{D^+(k)}(S^*,\Delta[-1])=\Hom_k(S^*,\CH^{-1}(\Delta))=\Hom_k(S^*,\bk[X]^{\times}/\bk^{\times})=0$$
et donc $\Hom_k(S^*,S_X^*) $ est isomorphe \`a 
$$\Hom_{D^+(k)}(S^*,S_X^*) \iso \Hom_{D^+(k)}(S^*,\Delta_n)\iso \Hom_{D^+(k)}(S^*,\Delta)\stackrel{\sim}{\leftarrow}  \Hom_{D^+(k)}(S^*,\tau_{\leq 0}\Delta) .$$
Alors $\chi([Y])\in \Hom_{D^+(k)}(S^*,\tau_{\leq 0}\Delta)$ donne un homomorphisme $\phi^*\in \Hom_k(S^*,S_X^*)$, et donc $\chi([Y])=\psi_{\leq 0}\circ \phi^*$.
Soit $\phi: S_X\to S$ l'homomorphisme correspondant. La suite exacte (\ref{BiD2e1}) implique l'\'enonc\'e.
\end{proof}

L'homomorphisme $\phi$ dans la proposition \ref{BiDprop2.1} est appel\'e \emph{le n-type de $[Y]$}.

Soit $\CT_X$ le torseur universel de $n$-torsion pour $X_{k_s}$, on obtient un isomorphisme de $\Gamma_k$-modules:
\begin{equation}\label{BiDprop2.1e3}
\tau_{X,S}: \Hom_{k_s}(S_X,S)\cong Hom_{k_s}(S^*,S_X^*)\to H^1(X_{k_s},S): \phi\mapsto \phi_*([\CT_X]).
\end{equation}
En particulier, on a deux $\Gamma_k$-isomorphismes naturels
\begin{equation}\label{BiDprop2.1e}
\tau_{X}:=\tau_{X,\mu_n}: S_X^*\iso H^1(X_{k_s},\mu_n): \phi\mapsto \phi_* (\CT_X)
\end{equation}
et
\begin{equation}\label{BiDprop2.1e1}
\tau_{X}(-1):=\tau_{X,\BZ/n}: \Hom_{k_s}(S_X,\BZ/n)\iso H^1(X_{k_s},\BZ/n): \phi\mapsto \phi_* (\CT_X).
\end{equation}
En fait, par d\'efinition, $\tau_X$ est exactement l'homomorphisme $S_X^*\to \CH^0(\Delta_n)$ induit par $\psi_{\leq 0}$.

Rappelons que, pour tous $F_1,F_2\in D^b(k)$, le produit tensoriel $F_1\otimes^LF_2$ est bien d\'efini et
 \emph{le cup-produit} est le homomorphisme canonique
$$\cup_j: \oplus_{r+s=j}  \CH^r(F_1)\otimes \CH^s(F_2) \to \CH^j(F_1\otimes^LF_2)$$
induit par la suite spectrale de Godement  (cf. \cite[Lem. VI. 8.6]{Mi80} ou \cite[Prop. 6.4.12]{Fu}).
De plus, $Rp_*\mu_n\in D^b(k)$ (\cite[Cor. 7.5.6]{Fu}).

\begin{cor}\label{BiDcor2.1.3}
Supposons que $k$ est s\'eparablement clos.
Soit $p:X\to \Spec\ k$ une vari\'et\'e int\`egre. Alors

(1) le cup-produit $\cup: H^1(X,\mu_n)\otimes \Hom(\mu_n,S)\to H^1(X,S): (\alpha,\varphi)\mapsto \varphi_*(\alpha)$
 est un isomorphisme;
 
 (2) pour tout complexe de $\BZ/n$-modules (vus comme $k$-faisceaux) de type fini $F$ avec $F\in D^{\geq 0}(k)\cap D^b(k)$,
on a $Rp_*\mu_n\otimes^L F\in D^{\geq 0}(k)$ et le cup-produit 
 $$\cup_j(F): \oplus_{r+s=j} R^rp_*\mu_n\otimes \CH^s(F)\cong \oplus_{r+s=j}  \CH^r(Rp_*\mu_n)\otimes \CH^s(F) \to \CH^j(Rp_*\mu_n\otimes^L F) $$
  est un isomorphisme pour $j=0$ et $j=1$;
 
 (3) dans (2), si $\CH^0(F)$ est plat, alors $\cup_2(F)$ est un isomorphisme.
\end{cor}

\begin{proof}
Puisque $X(k)\neq\emptyset$, il existe un torseur universel de $n$-torsion $\CT_X\to X$. D'apr\`es (\ref{BiDprop2.1e3}), on a le diagramme
$$\xymatrix{\Hom(S_X,\mu_n)\otimes \Hom(\mu_n,S) \ar[r]_-{\cong}^-{-\circ -} \ar[d]^{\tau_X\otimes id}_{\cong} &\Hom(S_X,S)\ar[d]^{\tau_{X,S}}_{\cong}\\
H^1(X,\mu_n)\otimes \Hom(\mu_n,S)\ar[r]^-{\cup} &H^1(X,S),
}$$
o\`u $-\circ -: (\psi,\phi)\mapsto \phi\circ \psi$. 
Ce diagramme est commutatif car $$\tau_{X,S}(\varphi\circ \phi)=(\varphi\circ \phi)_*[\CT_X ] = \varphi_*(\phi_*[\CT_X])  \stackrel{(\ref{BiDprop2.1e})}{=} \tau_X[\phi]\cup \varphi$$ 
pour tout $\phi\in \Hom(S_X,\mu_n)$ et tout $\varphi\in  \Hom(\mu_n,S)$. Donc on a (1).

Pour tout complexe $F$ dans (2), puisque la dimension cohomologique de $Rp_*$ est finie  (\cite[XIV]{SGA4}, cf. \cite[Cor. 7.5.6]{Fu}), on a:

(i) d'apr\`es \cite[XVII. Thm. 5.2.11]{SGA4}, pour tout $j<0$, on a $\CH^j(Rp_*\mu_n\otimes^L F)=0$ et donc $Rp_*\mu_n\otimes^L F\in D^{\geq 0}(k)$.
Ceci implique le premier \'enonc\'e de (2);

(ii) si $F\cong \CH^0(F)$ avec $\CH^0(F)$ plat, on a $\CH^j(Rp_*\mu_n\otimes^LF)=\CH^j(Rp_*\mu_n)\otimes F$ et donc $\cup_j(F)$ est un isomorphisme pour tout $j$;

(iii) si $F\cong \CH^0(F)$, on a  $Rp_*\mu_n\otimes^LF\cong Rp_*(\mu_n\otimes p^*F)$  (\cite[XVII. (5.2.11.1)]{SGA4}, cf. \cite[Cor. 6.5.6]{Fu}).
Puisque $X$ est int\`egre, $\cup_0(F): \mu_n\otimes F\to R^0p_*(\mu_n\otimes p^*F)$ est un isomorphisme, et d'apr\`es l'\'enonc\'e (1) et \cite[Prop. V.1.20]{Mi80}, le cup-produit
$$\cup_1(F): R^1p_*\mu_n\otimes F\cong H^1(X,\mu_n)\otimes \Hom(\mu_n,\mu_n\otimes F)\xrightarrow{\cup}H^1(X,\mu_n\otimes F)\cong \CH^1(Rp_*(\mu_n\otimes p^*F) ) $$
est un isomorphisme.
Ceci vaut seulement pour le cup-produit de degr\'e $1$.

Pour tout complexe $F$ dans (2), notons $F_+:=(\tau_{\geq 1}F)[1]$ un objet dans $D^{\geq 0}(k)\cap D^b(k)$. 
Alors $F_+$ v\'erifie toutes les hypoth\`eses dans (2).

Le triangle $\CH^0(F)\to F\to F_+[-1]\xrightarrow{+1}$ donne un diagramme commutatif de suites exactes:
\begin{equation}\label{BiDcor2.1.3e1} 
 \xymatrixcolsep{0.8pc}
 \xymatrix{0 \ar[r] & R^j(\CH^0(F))\ar[r]\ar[d]^{\cup_j(\CH^0(F))}  & 
 \oplus_{r+s=j} R^r( \CH^s(F))\ar[r]\ar[d]^{\cup_j(F)} & \oplus_{r+s=j-1} R^r(\CH^s(F_+)) \ar[r]\ar[d]^{\cup_{j-1}(F_+)}& 0\ar[d]\\
H_{\mu_n}^{j-2}(F_+)\ar[r]^-{\theta_{j-1}}&H_{\mu_n}^j(\CH^0(F))\ar[r] &H_{\mu_n}^j( F)\ar[r]&H_{\mu_n}^{j-1}(F_+)\ar[r]^-{\theta_j}  &H_{\mu_n}^{j+1}(\CH^0(F))
}\end{equation}
o\`u $H_{\mu_n}^j(-):=\CH^j(Rp_*\mu_n\otimes^L-)$, $R^{r}(-):=R^rp_*\mu_n\otimes -$ et la premi\`ere ligne est exacte car elle est scind\'ee.

Montrons l'\'enonc\'e (2).
D'apr\`es (i), on a: $H_{\mu_n}^{-2}(F_+)=H_{\mu_n}^{-1}(F_+)=0$.
D'apr\`es (iii), $\cup_0(\CH^0(F))$ est un isomorphisme. 
 Le lemme des cinq implique: $\cup_0(F)$ est un isomorphisme.
 Ceci donne l'\'enonc\'e(2) pour $j=0$.
 Donc  $\cup_0(F_+)$ est un isomorphisme. 
 D'apr\`es (iii), $\cup_1(\CH^0(F))$ est un isomorphisme. 
 Le lemme des cinq implique: $\cup_1(F)$ est un isomorphisme.

Montrons l'\'enonc\'e (3).
Par hypoth\`ese, $\CH^0(F)$ est plat, et d'apr\`es (ii), $\cup_2(\CH^0(F))$ est un isomorphisme.
D'apr\`es (2), $\cup_1(F_+)$ et $\cup_0(F_+)$ sont des isomorphismes.
 Donc $\theta_1=0$ et le lemme des cinq implique: $\cup_2(F)$ est un isomorphisme.
\end{proof}

Si $X$ est lisse, d'apr\`es (\ref{BiDprop2.2e}), l'isomorphisme (\ref{BiDprop2.1e1}) donne un $\Gamma_k$-isomorphisme naturel
\begin{equation}\label{BiDprop2.1e2}
\tau_{X}(-1): \Hom(S_X,\BZ/n)\iso H^1(X_{k_s},\BZ/n)\cong \Hom_{cont}(\pi_1(X_{k_s})^{ab},\BZ/n):\   \psi\mapsto \psi(k_s) \circ \tau_{\pi_1},
\end{equation}
o\`u $\tau_{\pi_1}: \pi_1(X_{k_s})^{ab}\to S_X(k_s)$ est l'homomorphisme induit par $\CT_X$. 
Ainsi $\tau_{\pi_1}$ induit un isomorphisme de $\Gamma_k$-modules $\pi_1(X_{k_s})^{ab}/n\iso S_X(k_s)$ et $\CT_X$ est g\'eom\'e\-triquement int\`egre.

\begin{cor}\label{BiDcor2.1.1}
Soit $X$ une vari\'et\'e lisse g\'eom\'e\-triquement int\`egre.
Soient $M$ un $\BZ/n$-module 
et $\pi_1(X_{k_s})^{ab}\xrightarrow{\theta} M$ un homomorphisme surjectif de noyau $\Gamma_k$-invariant.
Supposons qu'il existe un torseur universel de $n$-torsion pour $X$.
Alors il existe un $k$-groupe fini commutatif $S$ et un $S$-torseur $\CT\to X$ tels que
$\CT$ soit lisse g\'eom\'e\-triquement int\`egre, $S(k_s)=M$ et 
que, dans $H^1(X_{k_s},S)\cong \Hom_{cont}(\pi_1(X_{k_s})^{ab},M)$, on ait $[\CT_{k_s}]=\theta$.
\end{cor}

\begin{proof}
Soit $\CT_X$ un torseur universel de $n$-torsion pour $X$ (un torseur sous le $k$-groupe $S_X$).
Puisque $\Ker(\theta)$ est $\Gamma_k$-invariant, 
il existe une unique $\Gamma_k$-structure sur $M$ telle que $\theta$ soit un $\Gamma_k$-morphisme.
Ceci induit un $k$-groupe commutatif $S$ et un homomorphisme surjectif $\theta': S_X\to S$ 
tels que $S(k_s)=M$ et que $\theta'(k_s)\circ \tau_{\pi_1}=\theta$. 
Alors $\CT:=\theta'_*\CT_X:=\CT_X\times^{S_X}S $ donne l'\'enonc\'e.
\end{proof}

\bigskip

Soient $U$, $V$ deux vari\'et\'es g\'eom\'e\-triquement int\`egres sur $k$.
On consid\`ere le diagramme commutatif
\begin{equation}\label{BiDdiagprodu}
\xymatrix{U\times_kV \ar[r]^-{p_2}\ar[d]^{p_1} & V\ar[d]^{q_2}\\
U\ar[r]^-{q_1}& \Spec\ k.
}
\end{equation}

Soient $M$, $N$ deux $\BZ/n$-faisceaux finis plats sur le grand site de $k$.
  Le cup-produit donne un quasi-isomorphisme (\cite[Th. finitude Cor. 1.11]{SGA4.5}, cf. \cite[Cor. 9.3.5]{Fu}):
 \begin{equation}\label{BiDdiagprodue1}
 \cup: \ Rq_{1,*}M\otimes^L Rq_{2,*}N \cong  R(q_1\circ p_1)_* (M\otimes^L N).
 \end{equation}
Ceci induit le cup-produit (\cite[Prop. 6.4.12]{Fu}):
$$\cup_j: \oplus_{r+s=j} R^rq_{1,*}M\otimes_{\BZ/n} R^sq_{2,*}N \to \CH^j(Rq_{1,*}M\otimes^L Rq_{2,*}N) \iso R^j(q_1\circ p_1)_* (M\otimes^L N).$$

 \begin{lem}\label{BiDlem2.3.3}
 Le cup-produit $\cup_j$ est un isomorphisme pour $j=0,1,2$.
 \end{lem}
  
  \begin{proof}
  On peut supposer que $k$ est s\'eparablement clos. 
  Les $\BZ/n$-modules finis $M$, $N$ sont plats et donc ils sont des facteurs directs de $(\BZ/n)^{\oplus i}$ pour $i$ assez grand.  
  Puisque tous les foncteurs ci-dessus commutent avec  les sommes directs finies,  on peut supposer que $M=N=\mu_n$.
L'\'enonc\'e d\'ecoule du corollaire \ref{BiDcor2.1.3} (3) et de (\ref{BiDdiagprodue1}).
 \end{proof}

 Le r\'esultat ci-dessous g\'en\'eralise \cite[Thm. 2.6]{SZ}.
 
\begin{prop}\label{BiDlem2.3.1}
Supposons que $k$ est s\'eparablement clos.
Soient $U$, $V$ deux vari\'et\'es g\'eom\'e\-triquement int\`egres et $F$ un $\BZ/n$-module fini plat.
On consid\`ere le diagramme (\ref{BiDdiagprodu}).
Alors on a des isomorphismes naturels:
$$(p_1^*,p_2^*):\ H^1(U,F)\oplus H^1(V,F)\iso H^1(U\times V,F)$$
et 
$$(p_1^*,\cup,p_2^*):
 H^2(U,F)\oplus [H^1(U, \BZ/n)\otimes_{\BZ}H^1(V,F)]\oplus H^2(V,F)\iso H^2(U\times V,F),$$
 o\`u $\cup: H^1(U, \BZ/n)\otimes_{\BZ}H^1(V,F) \to H^2(U\times V,F)$ est le cup-produit.
\end{prop}

C'est clair que si $U,V$ sont d\'efinis sur un sous-corps $k_0\sbt k$ avec $k/k_0$ galoisienne et $F$ un $\Gal(k/k_0)$-module, 
alors les deux isomorphismes ci-dessus sont des isomorphismes de $\Gal(k/k_0)$-modules.

Si $\mathrm{char}(k)=0$, cette proposition d\'ecoule de  \cite[Prop. 2.2]{SZ} et de \cite[Thm. III.3.12]{Mi80} 
(on peut v\'erifier que l'homomorphisme dans \cite[Prop. 2.2]{SZ} est compactible avec le cup-produit).

\begin{proof}
On applique le lemme \ref{BiDlem2.3.3} au cas $U\times k$ et au cas $k\times V$, et on obtient deux diagrammes commutatifs:
$$\xymatrix{H^i(U,\BZ/n)\otimes H^0(k,F)\ar[r]_-{\cong}^-{\cup_U}\ar[d]^{\cong}_{id\times q_2^*}& H^i(U,F)\ar[d]^{p_1^*} & 
H^0(k,\BZ/n)\otimes H^i(V,F)\ar[r]_-{\cong}^-{\cup_V}\ar[d]^{\cong}_{q_1^*\times id}& H^i(V,F)\ar[d]^{p_2^*}\\
H^i(U,\BZ/n)\otimes H^0(V,F)\ar[r]^-{\cup_i} &H^i(U\times V,F)&   H^0(U,\BZ/n)\otimes H^i(V,F)\ar[r]^-{\cup_i} &H^i(U\times V,F) 
}$$
pour $i=1$ et $2$, o\`u $\cup_U$ (resp. $\cup_V$) est le cup-produit sur $U$ (resp. $V$).
Donc $$p_1^*(H^i(U,F) )=\cup_i( H^i(U,\BZ/n)\otimes H^0(V,F))\ \ \ \text{et}\ \ \ p_2^*(H^i(V,F) )=\cup_i( H^0(U,\BZ/n)\otimes H^i(V,F)).$$
L'\'enonc\'e d\'ecoule du lemme \ref{BiDlem2.3.3}.
\end{proof}

Soient $\CT_U$ (resp. $\CT_V$) un torseur universel de $n$-torsion pour $U$ (resp. pour $V$) et $S_U$ (resp. $S_V$) le groupe correspondant (cf. D\'efinition \ref{BiDdef2.20}).
Skorobogatov et Zarhin introduisent un homomorphisme (\cite[\S 5]{SZ}):
\begin{equation}\label{BiDvarepsilion}
\varepsilon: \Hom_k(S_U,S_V^*)\to H^2(U\times V,\mu_n): \phi\mapsto \phi_*[\CT_U]\cup [\CT_V],
\end{equation}
o\`u $\cup$ est le cup-produit $H^1(U,S_V^*)\times H^1(V, S_V)\to H^2(U\times V,\mu_n)$.
Les isomorphismes $\tau_V$ dans (\ref{BiDprop2.1e}) et $\tau_U(-1)$ dans (\ref{BiDprop2.1e1}) donnent un diagramme:
\begin{equation}\label{BiDvarepsiliondia}
\xymatrix{(\Hom_{k_s}(S_U,\BZ/n)\otimes S_V^*)^{\Gamma_k}\ar[r]^-{=}_-{\Phi}\ar[rd]^{\sim}_{(\tau_U(-1),\tau_V)}&
\Hom_k(S_U,S_V^*)\ar[r]^{\varepsilon}&H^2(U\times V,\mu_n)\ar[d]\\
&(H^1(U_{k_s},\BZ/n)\otimes H^1(V_{k_s},\mu_n))^{\Gamma_k}\ar[r]^-{\cup}&H^2((U\times V)_{\bk},\mu_n),
}
\end{equation}
qui est commutatif parce que, pour tous $\varphi\in \Hom_{k_s}(S_U,\BZ/n)$ et $\phi\in S_V^*=\Hom_{k_s}(S_V,\mu_n)$, on note $\phi^*:=\Hom_{k_s}(\phi, \mu_n):  \BZ/n\to S_V^*$ le dual de $\phi$, 
et on a:
$$ \varepsilon(\Phi(\varphi\otimes \phi))=\varepsilon(\phi^*\circ \varphi)=(\phi^*)_*( \varphi_*[\CT_U])\cup [\CT_V]  \stackrel{(1)}{=} \varphi_*[\CT_U]\cup \phi_*[\CT_V] =\tau_U(-1)(\varphi)\cup \tau_V(\phi),$$
o\`u (1) d\'ecoule du diagramme commutatif
$$\xymatrix{ H^1(U\times V,S_V)\ar[d]^{\phi_*}\ar@{}[r]|-{\times}&H^1(U\times V, \Hom_{k_s}(S_V,\mu_n)) \ar[r]^-{\cup}& H^2(U\times V,\mu_n)\ar[d]^=  \\
H^1(U\times V,\mu_n)\ar@{}[r]|-{\times} & H^1(U\times V,  \Hom_{k_s}(\mu_n ,\mu_n))\ar[r]^-{\cup}\ar[u]^{(\phi^*)_*=\Hom_{k_s}(\phi,\mu_n)_*}   &         H^2(U\times V,\mu_n).
}$$

Si $U(k)\neq\emptyset$, alors il existe un torseur universel de $n$-torsion pour $U$.
Pour un point $u\in U(k)$,
notons 
$$H^i_u(U,\mu_n):=\Ker (H^i(U,\mu_n)\xrightarrow{u^*}H^i(k,\mu_n)).$$

\begin{cor}\label{BiDlem2.3.2}
Sous les notations et hypoth\`eses ci-dessus,
supposons que $U(k)\neq\emptyset$ avec $u\in U(k)$ et qu'il existe des torseurs universels de $n$-torsion $\CT_U$ pour $U$ (sous le groupe $S_U$) et $\CT_V$ pour $V$ (sous le groupe $S_V$).
Alors on a un isomorphisme:
$$H^2_u(U,\mu_n)\oplus H^2(V,\mu_n)\oplus \Hom_k(S_U,S_V^*)\xrightarrow{(p_1^*,p_2^*,\varepsilon) }H^2(U\times V,\mu_n).$$
\end{cor}

\begin{proof}
Notons $E_2^{i,j}(U):= H^i(k,H^j(U_{k_s},\mu_n))\Rightarrow H^{i+j}(U,\mu_n)$ la  suite spectrale de Hochschild-Serre de $U$ et $E_2^{i,j}(V)$ (resp. $E^{i,j}_2(U\times V)$) celle de $V$ (resp. de $U\times V$). 

Notons $H^i_u(U_{k_s},\mu_n):=\Ker (H^i(U_{k_s},\mu_n)\xrightarrow{u^*}H^i(k_s,\mu_n))$.
Alors $H^0_u(U_{k_s},\mu_n)=0$ et $H^i_u(U_{k_s},\mu_n)= H^i(U_{k_s},\mu_n)$ pour $i\neq 0$.
La suite spectrale de Hochschild-Serre donne canoniquement une suite spectrale:
$$E_2^{i,j}(U,u):= H^i(k,H^j_u(U_{k_s},\mu_n))\Rightarrow H^{i+j}_u(U,\mu_n).$$
Soit $\phi^{i,j}_2: E_2^{i,j}(U,u)\oplus E_2^{i,j}(V)\to E^{i,j}_2(U\times V)$  le morphisme de suites spectrales induit par $(p_1^*,p_2^*)$.
D'apr\`es la proposition \ref{BiDlem2.3.1},  $\phi^{i,j}_2$ est un isomorphisme pour $j=0,1$
 et $\phi^{0,2}_2$ est injectif. 
Ainsi $\phi_2^{i,j}$ induit une suite exacte par le lemme des cinq:
$$0\to H^2_u(U,\mu_n)\oplus H^2(V,\mu_n)\xrightarrow{p_1^*,p_2^*}H^2(U\times V,\mu_n)\to \coker(\phi^{0,2}_2).$$
D'apr\`es  la proposition \ref{BiDlem2.3.1} et le diagramme (\ref{BiDvarepsiliondia}), on a
 $\coker(\phi^{0,2}_2)\cong (H^1(U_{\bk},\BZ/n)\otimes H^1(V_{\bk},\mu_n))^{\Gamma_k}$
 et la composition  
 $$\Hom_k(S_U,S_V^*)\xrightarrow{\varepsilon} H^2(U\times V,\mu_n)\to H^2((U\times V)_{\bk},\mu_n)^{\Gamma_k}\to \coker(\phi^{0,2}_2) $$
  est un isomorphisme, d'o\`u le r\'esultat.
\end{proof}

\section{Pr\'eliminaires sur les torseurs sous un groupe fini}

Dans toute cette section,  $k$ est un corps quelconque de caract\'eristique $0$. 
Sauf  mention explicite du contraire,  une vari\'et\'e est une $k$-vari\'et\'e.

Soit $G$ un groupe alg\'ebrique connexe et $X$ une $G$-vari\'et\'e lisse g\'eom\'e\-triquement int\`egre.
Cette section traite trois probl\`emes: 
pour un torseur $H\to G$ sous un $k$-groupe fini, on montre l'existence et l'unicit\'e de la structure de groupe sur $H$ dans \S \ref{3.1};
pour un torseur $Y\to X$ sous un $k$-groupe fini, on donne dans \S \ref{3.2} 
une condition n\'ecessaire et suffisante pour le rel\`evement, de fa\c{c}on compatible, de l'action de $G$ sur $X$ en une action sur $Y$;
si ce rel\`evement n'existe pas, on montre dans \S \ref{3.3} l'existence d'une isog\'enie minimale $H_Y\to G$ telle que l'action de $H_Y$ puisse \^etre relev\'ee en une action sur $Y$.

\subsection{Torseur sur un groupe alg\'ebrique}\label{3.1}

Pour un  groupe alg\'ebrique connexe $G$, tout recouvrement \'etale fini de $G_{\bk}$ est une extension centrale de $G_{\bk}$ (Brion et Szamuely \cite[Prop. 1.1 (1)]{BS}).
Le r\'esultat suivant g\'en\'eralise ce r\'esultat au corps de base et il est aussi un analogue d'un r\'esultat de Colliot-Th\'el\`ene (\cite[Thm. 5.6]{CT07}).

\begin{prop}\label{BiDcorprop2.1}
Soit $G$ un groupe alg\'ebrique connexe, $S$ un $k$-groupe fini commutatif et $\psi: H\to G$ un $S$-torseur avec $H$ g\'eom\'e\-triquement int\`egre sur $k$.
S'il existe un point $e_H\in H(k)$ avec $\psi(e_H)=e_G$, alors il existe une unique structure de $k$-groupe alg\'ebrique sur $H$
 telle que $\psi$ soit un homomorphisme et que $e_H$ soit l'unit\'e.
 
 De plus, dans ce cas, $\Ker(\psi)=S$ et l'action de $S$ sur $H$ est compatible avec la multiplication de $H$.
\end{prop}

\begin{proof}
 L'existence d'une structure de groupe sur $H$ est \'equivalente \`a l'existence d'un couple de morphismes $(m_H,i_H)$  satisfaisant certaines relations
o\`u $m_H: H\times H\to H$ est la multiplication et $i_H: H\to H$ est l'inverse.

Pour l'unicit\'e, s'il existe deux structures de groupe sur $H$, soient $(m_H,i_H)$, $(m'_H,i'_H)$ les couples de morphismes correspondants. 
Soient $m_G$ la multiplication de $G$ et $i_G$ l'inverse de $G$.
Alors $$\psi\circ m_H=m_G\circ (\psi\times \psi)=\psi\circ m'_H,\ \ \ m_H(e_H\times e_H)=e_H=m'_H(e_H\times e_H)$$
et
$$\psi\circ i_H=i_G\circ \psi=\psi\circ i'_H,\ \ \ i_H(e_H)=e_H=i'_H(e_H).$$
Puisque $\psi$ est fini \'etale et $H\times H$ est int\`egre, on a $m_H=m'_H$ et $i_H=i'_H$ (\cite[Cor. I.3.13]{Mi80}). 
 Ceci donne l'unicit\'e de $(m_H,i_H)$.

Pour l'existence de la structure de groupe (i.e. l'existence de $(m_H, i_H)$), par la descente galoisienne et l'unicit\'e de $(m_H,i_H)$, il suffit d'\'etablir l'existence de $(m_H,i_H)$ sur $\bk$.
On peut supposer que $k=\bk$.
Dans ce cas, $\psi$ est fini \'etale galoisien avec $\Aut(H/G)\cong S(\bk)$.
D'apr\`es \cite[Prop. 1.1 (1)]{BS}, il existe une structure de groupe sur $H$ telle que $\psi: H\to G$ soit une isog\'enie centrale. 
Notons $-\cdot -$ la multiplication et $(-)^{-1}$ l'inverse de cette structure de groupe. 
Soit $c:=e_H\cdot e_H$ et $d: =e_H\cdot c^{-1}$.
Les points $e_H$, $c$ et $d$ sont dans $\Ker(\psi)$ et donc dans le centre de $H$.
Alors les morphismes $m'_H: H\times H\to H: (h_1,h_2)\mapsto d\cdot h_1\cdot h_2$ et $i'_H: H\to H: h\mapsto c\cdot h^{-1}$ d\'efinissent sur
$H$ une nouvelle structure de groupe et cette structure v\'erifie les hypoth\`eses ci-dessus.

Pour le dernier \'enonc\'e, puisque $S\sbt \Ker(\psi)$,
l'action de $S$ induit une inclusion de $\Gamma_k$-module $S(\bk)\sbt \Aut(H_{\bk}/G_{\bk})$  
et la multiplication de $H$ induit une inclusion $\Ker(\psi)(\bk)\sbt \Aut(H_{\bk}/G_{\bk})$ de $\Gamma_k$-module.
Puisque $\# \Aut(H_{\bk}/G_{\bk})=\deg(\psi)$, les deux inclusions ci-dessus sont isomorphes, d'o\`u le r\'esultat.
\end{proof}

\begin{cor}\label{BiDcor2.1}
Soit $G$ un groupe alg\'ebrique connexe.
Pour tout $\BZ/n$-module fini $M$ et 
tout homomorphisme surjectif $\pi_1(G_{\bk})\xrightarrow{\theta} M$ de noyau $\Gamma_k$-invariant, 
il existe un unique groupe alg\'ebrique connexe $H$ isog\`ene \`a $G$, i.e. muni d'un homomorphisme fini surjectif $\psi: H\to G$,
tel que $(\Ker(\psi))(\bk)\cong M$ et que 
la composition $\pi_1(H_{\bk})\xrightarrow{\psi_{\pi_1}} \pi_1(G_{\bk})\xrightarrow{\theta} M$ soit nulle.

De plus, pour tout groupe alg\'ebrique connexe $H_1$, tout homomorphisme fini surjectif $\psi_1: H_1\to G $ v\'erifiant $\theta\circ \psi_{1,\pi_1}=0$
se factorise par $\psi$.
\end{cor}

\begin{proof}
Puisque $G(k)\neq\emptyset$, il existe un unique torseur universel de $n$-torsion $\CT_G$ (un $S_G$-torseur sur $G$)
tel que $\CT_G|_{e_G}\cong S_G$.

D'apr\`es le corollaire \ref{BiDcor2.1.1}, il existe un $k$-groupe fini commutatif $S$ et un $S$-torseur $H\xrightarrow{\psi}G$
tels que $S(\bk)=M$ et que l'homomorphisme $\pi_1(G_{\bk})\to S(\bk) $ induit par $[H_{\bk}]$ soit $\theta$.
 Donc la composition $\theta\circ \psi_{\pi_1}$ est nulle.
 Apr\`es avoir tordu par un \'el\'ement de $H^1(k,S)$, on peut supposer que $[H]|_{e_G}=0\in H^1(k,S)$.
D'apr\`es la proposition \ref{BiDcorprop2.1}, il existe une structure de groupe sur $H$
 telle que $\psi$ soit un homomorphisme et que $\Ker(\psi)=S$.
 
 Pour tout groupe alg\'ebrique connexe $H_1$ et tout homomorphisme fini surjectif $\psi_1: H_1\to G $, 
le noyau $\psi_1$ est commutatif et on a une suite exacte de $\Gamma_k$-modules:
 $$\pi_1(H_{1,\bk})\to \pi_1(G_{\bk})\to \Ker(\psi_1)(\bk)\to 0.$$
Ceci donne un homomorphisme surjectif de $\Gamma_k$-modules $\theta_1:\Ker(\psi_1)(\bk)\to M\cong S(\bk) $ et,
puisque $[H_1]|_{e_G}=0=[H]|_{e_G}$, on a $\theta_{1,*}([H_1])=[H]\in H^1(X,S)$.
En utilisant l'action de $S$, on a  un $\Ker(\psi_1)$-morphisme $\phi: H_1\to H $ au-dessus de $G$ tel que $\phi(e_{H_1})=e_H$.
Soient $\chi_1,\chi_2: H_1\times H_1\to H $ deux morphismes
 avec $\chi_1(h_1,h_2)=\phi(h_1\cdot h_2)$ et $\chi_2(h_1,h_2)=\phi(h_1)\cdot \phi(h_2)$ pour tous $h_1,h_2\in H_1$.
 Alors $\chi_1(e_{H_1},e_{H_1})=\chi_2(e_{H_1},e_{H_1})$ et $\psi\circ \chi_1=\psi\circ \chi_2 $.
 Ceci induit:
 $$\chi: H_1\times_kH_1\xrightarrow{\chi_1,\chi_2} H\times_GH\cong H\times_kS\xrightarrow{p_2} S.$$
Puisque $H_1$ est connexe, on a $\Im (\chi)=e_S$, $\chi_1=\chi_2$ et $\phi$ est un homomorphisme.
\end{proof}

\subsection{Rel\`evement d'une action par un torseur}\label{3.2}

Soient $G$ un groupe alg\'ebrique connexe et $(X,\rho )$ une $G$-vari\'et\'e lisse g\'eom\'e\-triquement int\`egre.
Soient $F$ un $k$-groupe fini et $f: Y\to X$ un $F$-torseur.
Notons $p_1: G\times X\to G$, $p_2: G\times X\to X$ les deux projections.

D'apr\`es \cite[X.2.2]{SGA1}, on a deux suites exactes de groupes fondamentaux
$$1\to \pi_1(X_{\bk})\to \pi_1(X)\to \Gamma_k\to 1\ \ \ \text{et}\ \ \ 1\to \pi_1((G\times X)_{\bk})\to \pi_1(G\times X)\to \Gamma_k\to 1.$$
D'apr\`es \cite[XII.5.2]{SGA1}, on a $\pi_1((G\times X)_{\bk})\cong \pi_1(G_{\bk})\times \pi_1(X_{\bk})$, 
car ceci vaut pour les espaces topologiques.
Alors on a une suite exacte de groupes fondamentaux:
$$1\to \pi_1(G_{\bk})\to \pi_1(G\times X)\xrightarrow{p_{2,\pi_1}} \pi_1(X)\to 1 $$
qui admet une section induite par $i_e: X\to G\times X: x\mapsto (e_G,x) $ et l'action de $ \pi_1(X)$ sur $\pi_1(G_{\bk})$ se factorise par $\Gamma_k$. 
D'apr\`es (\ref{BiDprop2.2e}), cette suite exacte induit une suite exacte d'ensembles point\'es (voir \cite[\S 5.8]{Se})
\begin{equation}\label{BiDprop2.2e1}
1\to H^1(X,F)\xrightarrow{p_2^*} H^1(G\times X,F)\xrightarrow{\iota} H^1(G_{\bk},F)^{\Gamma_k}
\end{equation}
et $p_2^*$ admet une section induite par $i_e^*$.

\begin{prop}\label{BiDprop2.2}
Soient $G$ un groupe alg\'ebrique connexe et $(X,\rho )$ une $G$-vari\'et\'e lisse g\'eom\'e\-triquement int\`egre.
Soient $F$ un $k$-groupe fini et $f: Y\to X$ un $F$-torseur.
Alors les hypoth\`eses ci-dessous sont \'equivalentes:

(a) on a $\rho^*([Y])=p_2^*([Y])\in H^1(G\times X,F) $;

(b) pour $\iota$ dans (\ref{BiDprop2.2e1}), on a $\iota (\rho^*([Y]))=0\in H^1(G_{\bk},F)$;

(c) le $F$-torseur $Y$ est $G$-compatible, i.e. l'action de  $G$ sur $X$ se rel\`eve en une action sur $Y$;

(d) il existe un morphisme $\rho_Y: G\times Y\to Y$  tel que $\rho_Y|_{e_G\times Y}=id_Y$
et que $\rho_Y$ soit compatible avec $\rho$, i.e. $\rho\circ (id_G\times f)=f\circ \rho_Y$.

De plus, sous les hypoth\`eses ci-dessus, on a

(1) l'action de $G$ sur $Y$ pour laquelle $f$ est un $G$-morphisme est unique;

(2) l'action de $G$ et celle de $F$ commutent;

(3) pour tout $\sigma\in H^1(k,F)$, le $F_{\sigma}$-torseur $Y_{\sigma}$ est $G$-compatible.
\end{prop}

\begin{proof}
Puisque $i_e^*(p_2^*([Y]))=i_e^*(\rho^*([Y]))$ dans $H^1(X,F)$,
l'\'equivalence (a)$\Leftrightarrow$(b) d\'ecoule de la suite exacte (\ref{BiDprop2.2e1}).

\begin{lem}\label{BiDlem2.2}
Pour tout $k$-sch\'ema de type fini $Z$ et tous morphismes $\theta_1,\theta_2: G\times Z\to Y$,
si $f\circ \theta_1=f\circ \theta_2$ et $\theta_1|_{e_G\times Z}=\theta_2|_{e_G\times Z}$, alors $\theta_1=\theta_2$.
\end{lem}

\begin{proof}
En fait, $\theta_1$, $\theta_2$ induisent un morphisme 
$$\theta: G\times Z\xrightarrow{(\theta_1,\theta_2)} Y\times_XY\cong Y\times_kF\xrightarrow{pr_F}F$$
tel que $\theta (e_G\times Z)= e_F$.
Puisque $G$ est int\`egre, $\theta(G\times Z)= e_F$ et donc $\theta_1=\theta_2$.
\end{proof}

Pour (a)$\Rightarrow $(d),
soient $\rho^*Y$ le  pullback de $Y$ par $\rho$ et $p_2^*Y:=G\times Y$.
Notons $\Mor_F(p_2^*Y,F)$ l'ensemble des morphismes $\chi:p_2^*Y\to F $ 
tels que $\chi(a\cdot y)=a\cdot \chi(y)\cdot a^{-1}$ pour tous $a\in F$ et $y\in p_2^*Y$. 
D\'efinissons de m\^eme $\Mor_F(Y,F)$.
Alors $Y\cong e_G\times Y\sbt p_2^*Y$ induit un morphisme surjectif $\Mor_F(p_2^*Y,F)\xrightarrow{\Mor(i_e)} \Mor_F(Y,F)$,
car il existe une section induite par $p_2$.
Par hypoth\`ese, on a un isomorphisme de $F$-torseur $p_2^*Y\xrightarrow{\phi} \rho^*Y $.
Pour tout isomorphisme $\phi_1$,
 l'argument classique montre qu'il existe un $\chi_1\in \Mor_F(p_2^*Y,F)$ tel que $\phi_1=\chi_1\cdot \phi$.
Puisque $\Mor(i_e)$ est surjectif, on peut supposer que $\phi|_{e_G\times X}$ est l'identit\'e de $Y$.
Le morphisme $\rho_Y: G\times Y\xrightarrow{\phi}\rho^*Y\to Y$ donne (d).

Pour (d)$\Rightarrow $(c), l'hypoth\`ese (d) donne un diagramme commutatif:
\begin{equation}\label{BiDprop2.2diagram}
\xymatrix{id_Y: &e_G\times Y\ar[d]^f\ar[r]^{i_{e,Y}}&G\times Y\ar[d]^{id_G\times f}\ar[r]^{\rho_Y}&Y\ar[d]^f\\
id_X:&e_G\times X\ar[r]^{i_e}&G\times X\ar[r]^{\rho}&X
}
\end{equation}
tel que $\rho_Y\circ i_{e,Y}=id_Y$.
Soient $\theta_1,\theta_2: G\times G\times Y\to Y$ les deux morphismes d\'efinis par
 $$\theta_1(g_1,g_2,y)=g_1\cdot (g_2\cdot y)\ \ \ \text{et}\ \ \  \theta_2(g_1,g_2,y)=(g_1\cdot g_2)\cdot y$$
  pour tous $g_1,g_2\in G$ et $y\in Y$.
Alors $\theta_1(e_G,g_2,y)=\theta_2(e_G,g_2,y)$ et le lemme \ref{BiDlem2.2} montre que $\theta_1=\theta_2$.
Donc $\rho_Y$ est une action et $f$ est un $G$-morphisme. Ceci donne (c).

\medskip

Supposons (c) et montrons (1), (2), (3) et (a). 

L'hypoth\`ese (c) donne aussi le diagramme commutatif (\ref{BiDprop2.2diagram}) avec $\rho_Y$ l'action relev\'ee de $G$ sur $Y$.

Soient $\theta_1, \theta_2$ deux actions de $G$ sur $Y$ telles que $f$ soit un $G$-morphisme.
Puisque $f\circ \theta_1=\rho\circ (id_G\times f)=f\circ \theta_2$,
On applique le lemme \ref{BiDlem2.2} \`a  $\theta_1,\theta_2: G\times Y\to Y$ et on obtient (1).

Soient $\theta_1, \theta_2: G\times F\times Y\to Y$ les deux morphismes d\'efinis par 
$\theta_1(g,a,y)=g\cdot (a\cdot y)$ et $\theta_2(g,a,y)=a\cdot (g\cdot y)$ pour tous $g\in G $, $a\in F$ et $y\in Y$.
Alors $\theta_1(e_G,a,y)=a\cdot y=\theta_2(e_G,a,y) $ et 
$(f\circ \theta_1)(g,a,y)=g\cdot f(y)=(f\circ \theta_2 )(g,a,y)$.
On applique le lemme \ref{BiDlem2.2} \`a  $\theta_1, \theta_2: G\times F\times Y\to Y$ et on obtient (2).

Pour le $F$-torseur $p_2^*([Y])=(G\times Y\to G\times X)$,
 l'\'enonc\'e (2) montre que l'action $G\times Y\to Y$ est un $F$-morphisme compatible avec $\rho$.
Ceci induit un isomorphisme de $F$-torseurs $p_2^*([Y])=\rho^*([Y])$ et on a (a).

Puisque l'\'enonc\'e (b) est un \'enonc\'e sur $\bk$, on obtient (3).
\end{proof}

\begin{cor}\label{BiDrem2.2}
Soient $G$ un groupe alg\'ebrique connexe et $(X,\rho )$ une $G$-vari\'et\'e lisse g\'eom\'e\-triquement int\`egre.
Alors $\rho$ induit un homomorphisme $\rho_{\pi_1}: \pi_1(G_{\bk})\to \pi_1(X)$ 
et, pour tout $k$-groupe fini $F$, il induit $\rho_{\pi_1}^*: H^1(X,F)\to H^1(G_{\bk},F)$
 et  on a:

(1) le sous-groupe $\Im(\rho_{\pi_1})\sbt \pi_1(X)$ est normal et il est contenu dans le centre de $\pi_1(X_{\bk})$;

(2) pour tout $\alpha\in H^1(X,F)$, on a $\rho_{\pi_1}^*(\alpha)=\iota (\rho^*(\alpha))$, o\`u $\iota$ est dans (\ref{BiDprop2.2e1});

(3) pour tout 1-cocycle $a$ de $\pi_1(X)$ \`a valeurs dans $F(\bk)$, l'homomorphisme $a\circ \rho_{\pi_1}: \pi_1(G_{\bk})\to F(\bk)$ est de noyau $\Gamma_k$-invariant, et il est nul si et seulement si $\rho_{\pi_1}^*([a])=0$;

 (4) si $X$ est un $G$-espace homog\`ene \`a stabilisateur g\'eom\'etrique connexe,
alors tout $F$-torseur $G$-compatible est constant, i.e. ce torseur est isomorphe \`a $M\times_k X$ avec $M$ un $F$-torseur sur $k$.
\end{cor}

\begin{proof}
L'\'enonc\'e (1) vaut car $\pi_1(G_{\bk})=\Ker(\pi_1(G\times X)\xrightarrow{p_{2,*}} \pi_1(X))$ et $\pi_1((G\times X)_{\bk})\cong \pi_1(G_{\bk})\times \pi_1(X_{\bk})$. 
Les \'enonc\'es (2) et (3) d\'ecoulent par d\'efinition.

Pour (4), dans ce cas, $\Im(\rho_{\pi_1})=\pi_1(X_{\bk})$ (\cite[Prop. 5.5.4]{Sz}).
D'apr\`es la proposition \ref{BiDprop2.2} et (2), (3) ci-dessus, tout $F$-torseur $G$-compatible est trivial sur $X_{\bk}$, 
et donc il provient d'un $F$-torseur sur $k$.
\end{proof}

\begin{cor}\label{BiDcor2.2.1}
Sous les notations et les hypoth\`eses ci-dessus, supposons que $f$ est $G$-compatible.
Alors, pour tout $k$-sch\'ema fini \'etale $E$, 
la restriction de Weil $V:=R_{X\times E/X}(Y\times E)$ est un $R_{E/k}(F\times_kE)$-torseur $G$-compatible sur $X$.
\end{cor}

\begin{proof}
Notons $f_V: V\to X$. Par hypoth\`ese, $f_V$ est un torseur sous le groupe 
$$R_{X\times E/X}(F\times X\times E)\cong R_{E/k}(F\times_kE).$$
On consid\`ere $G\times V$ comme un $X$-sch\'ema par le morphisme $G\times V\xrightarrow{id_G\times f_V}G\times X\xrightarrow{\rho}X $ et $G\times Y$ comme un $X$-sch\'ema par $\rho\circ (id_G\times f)$.
Dans ce cas, tout morphisme $\rho_V\in \Mor_X(G\times V,V)$ satisfait $f_V\circ \rho_V=\rho\circ (id_G\times f_V)$.
D'apr\`es la proposition \ref{BiDprop2.2} (d), il suffit de trouver un $\rho_V\in \Mor_X(G\times V,V)$ tel que $\rho_V|_{e_G\times V}=id_V$.
Puisque 
$$ \Mor_X(V,V)\iso \Mor_{X\times E}(V\times E,Y\times E)\ \ \ \text{et que}\ \ \  
 \Mor_X(G\times V,V)\iso \Mor_{X\times E}(G\times V\times E,Y\times E),$$
l'identit\'e $id_V$ induit un morphisme $V\times E\xrightarrow{\theta}Y\times E $.
Le $X\times E$-morphisme
$$ G\times V\times E\xrightarrow{id_G\times \theta}G\times Y\times E\xrightarrow{\rho_Y\times id_E}Y\times E$$
 induit un morphisme $\rho_V\in \Mor_X(G\times V,V)$ qui satisfait $\rho_V|_{e_G\times V}=id_V$.
\end{proof}

\begin{cor}\label{BiDcor2.2}
Soient $G$ un groupe alg\'ebrique connexe, $Z$ une vari\'et\'e lisse g\'eom\'e\-triquement int\`egre et $p: X\to Z$ un $G$-torseur.
Pour tout $k$-groupe fini $F$ et tout $F$-torseur $G$-compatible $Y\to X$, 
il existe un $F$-torseur $Y_Z$ sur $Z$ tel que $[Y]=p^*( [Y_Z])\in H^1(X,F)$.
\end{cor}

\begin{proof}
D'apr\`es la proposition \ref{BiDprop2.2} (2), $Y$ est un $G\times F$-torseur sur $Z$ tel que $Y/F=X$.
Alors $Y_Z:=Y/G$ est un $F$-torseur sur $Z$ et $Y\to Y_Z$ est un $F$-morphisme.
Donc $[Y]=p^*( [Y_Z])$.
\end{proof}

\subsection{Le groupe minimal compatible avec un torseur}\label{3.3}

Soit $G$ un groupe alg\'ebrique connexe.
Soit $\CC_G$ la cat\'egorie des groupes alg\'ebriques connexes $H$ isog\`enes \`a $G$, 
i.e. munis d'un homomorphisme fini surjectif $\psi: H\to G$.
C'est clair que si $G$ est lin\'eaire, tout objet dans $\CC_G$ est aussi lin\'eaire.

Soit $(X,\rho )$ une $G$-vari\'et\'e lisse g\'eom\'e\-triquement int\`egre.
Soient $F$ un $k$-groupe fini et $f: Y\to X$ un $F$-torseur.
Soit $\CC_G(Y)$ la sous-cat\'egorie pleine de $\CC_G$ 
dont les objets sont les groupes $H$ isog\`enes \`a $G$ tels que $f$ soit $H$-compatible.
D'apr\`es la proposition \ref{BiDprop2.2} (1),
tout objet $H\in \CC_G(Y)$ admet une unique action sur $Y$ telle que $f$ soit un $H$-morphisme.
Alors tout morphisme de $\CC_G(Y)$ est compatible avec les actions ci-dessus.

\begin{prop}\label{BiDprop2.3}
La cat\'egorie $\CC_G(Y)$ admet un objet final $(H_Y\xrightarrow{\psi_Y}G)$, et 
un objet $(H\xrightarrow{\psi}G)\in \CC_G(Y)$ est final si et seulement si l'action de $\ker(\psi)$ sur $Y$ est libre.
\end{prop}

\begin{proof}
Dans la suite exacte (\ref{BiDprop2.2e1}), notons $\alpha:=\iota (\rho^*([Y]))\in H^1(G_{\bk},F)^{\pi_1(X)}$.
   Soit $\theta\in \Hom_{cont}(\pi_1(G_{\bk}),F(\bk))$ un \'el\'ement correspondant \`a $\alpha$ selon (\ref{BiDprop2.2e}).
D'apr\`es le corollaire \ref{BiDrem2.2} (3), le noyau $\Ker(\theta)$ est $\Gamma_k$-invariant.

La fonctorialit\'e de (\ref{BiDprop2.2e1}) et  la proposition \ref{BiDprop2.2} montrent qu'un objet $(H\xrightarrow{\psi}G)\in \CC_G$ est contenu dans $\CC_G(Y)$ 
si et seulement si $\psi_*(\alpha)=0\in H^1(H_{\bk},F)$, 
 i.e. $\theta\circ \psi_{\pi_1}=0 $ (Corollaire \ref{BiDrem2.2} (3)), o\`u $\psi_{\pi_1}: \pi_1(H_{\bk})\to \pi_1(G_{\bk})$.
Puisque $\pi_1(G_{\bk})$ est ab\'elien (\cite[Thm. 1]{Miy}),
le corollaire \ref{BiDcor2.1} implique l'existence de l'objet final de $\CC_G(Y)$.
 
L'argument ci-dessus montre que la cat\'egorie $\CC_G(Y)$ est stable par changement de base,
 i.e., pour toute $G$-vari\'et\'e $X'$ et tout $G$-morphisme $X'\to X$, 
 on a un $F$-torseur $Y':=Y\times_XX'\to X'$ et $\CC_G(Y')=\CC_G(Y)$ comme sous-cat\'egories de $\CC_G$.
 
 Soit $(H_Y\xrightarrow{\psi_Y}G)$ l'objet final de $\CC_G(Y)$.
 Il est l'objet final de $\CC_G(Y')$ aussi pour tout $Y'\to X'$ ci-dessus.
 
 Pour montrer que l'action de $\Ker(\psi_Y)$ est libre, 
 on peut supposer que $k=\bk$ et que $X$ est un espace homog\`ene de $G$.
 Dans ce cas, $Y$ est un espace homog\`ene de $F\times H_Y$ (Proposition \ref{BiDprop2.2} (2)).
 Puisque $\Ker(\psi_Y)$ est dans le centre de $F\times H_Y$,
  les stabilisateurs de $\Ker(\psi_Y)$ en tous les points $x\in X$ sont les m\^emes.
 La propri\'et\'e de l'objet final implique que l'action de $\Ker(\psi_Y)$ soit libre.

Soit $(H\xrightarrow{\psi}G)\in \CC_G(Y)$ un objet tel que l'action de $\ker(\psi)$ sur $Y$ est libre.
Soit $\phi: H\to H_Y$ l'homomorphisme canonique.
Puisque $\psi_Y$, $\psi$ sont finis surjectifs et que $H_Y$ est connexe, l'homomorphisme $\phi$ est fini surjectif.
 La proposition \ref{BiDprop2.2} (1) implique que $\phi$ est compatible avec l'action de $H$ et de $H_Y$.
Puisque l'action de $\Ker(\psi)$ sur $Y$ est libre, $\phi$ est un isomorphisme.
\end{proof}

\begin{defi}\label{BiDdef2.1}
L'objet final $(H_Y\xrightarrow{\psi_Y}G)$ de $\CC_G(Y)$ est appel\'e \emph{le groupe minimal compatible avec le $F$-torseur $Y$}.
\end{defi}

\begin{rem}\label{BiDrem2.3}
Soit $\rho_{\pi_1}: \pi_1(G_{\bk})\to\pi_1(X)$ l'homomorphisme dans le corollaire \ref{BiDrem2.2} 
et soit $\alpha$ un 1-cocycle de $\pi_1(X)$ en $F(\bk)$ qui correspond \`a $[Y]\in H^1(X,F)$.
Alors $\alpha|_{\pi_1(X_{\bk})}$ est un homomorphisme.
Par la d\'emonstration de la proposition \ref{BiDprop2.3}, 
le groupe minimal compatible au $F$-torseur $Y$ est d\'etermin\'e par $\Ker(\alpha\circ \rho_{\pi_1})$,
o\`u $\alpha\circ \rho_{\pi_1}: \pi_1(G_{\bk})\to F(\bk)$ est un homomorphisme.
 Donc ceci est d\'etermin\'e par $\Ker(\alpha|_{\Im(\rho_{\pi_1})})$.
\end{rem}

D'apr\`es la proposition \ref{BiDprop2.2} (3), $H_Y$ est aussi le groupe minimal compatible au $F_{\sigma}$-torseur $Y_{\sigma}$ pour tout $\sigma\in H^1(k,F)$.

\begin{cor}\label{BiDcor2.3}
Sous les notations et les hypoth\`eses ci-dessus, si $Y$ est g\'eom\'e\-triquement int\`egre sur $k$, 
alors il existe un homomorphisme injectif $\phi: \Ker(\psi_Y)\to F $ d'image centrale  compatible avec l'action de $\Ker(\psi_Y)$ et de $F$ sur $Y$.
\end{cor}

\begin{proof}
L'action de $\Ker(\psi_Y)$ induit un morphisme: 
$$\Phi: \Ker(\psi_Y)\times Y\xrightarrow{\rho_{H_Y},pr_Y} Y\times_XY \iso F\times_kY \xrightarrow{pr_F}F,$$
o\`u $\rho_{H_Y}$ est l'action de $H_Y$.
Pour tous $h\in \Ker(\psi_Y), y\in Y$, on a $h\cdot y=\Phi(h,y)\cdot y$.

Puisque $Y$ est g\'eom\'e\-triquement int\`egre, il existe un morphisme $\phi: \Ker(\psi_Y)\to F $ tel que $\Phi=\phi\circ p_1$, 
o\`u $p_1: \Ker(\psi_Y)\times Y\to \Ker(\psi_Y)$ est la projection.
Puisque l'action de $F$ sur $Y$ est libre, $\phi$ est un homomorphisme.
La proposition \ref{BiDprop2.2} (2) implique que l'image de $\phi$ est centrale.
D'apr\`es la proposition \ref{BiDprop2.3}, l'action de $\Ker(\psi_Y)$ est libre et donc $\phi$ est injectif.
\end{proof}

Rappelons la d\'efinition de $\Br_G(X)$ dans la d\'efinition \ref{BiDdefBrainv}.

\begin{prop}\label{BiDcor2.3.1}
Soient $G$ un groupe alg\'ebrique connexe et $X$ une $G$-vari\'et\'e lisse g\'eom\'e\-triquement int\`egre.
Supposons qu'il existe un torseur universel de $n$-torsion $\CT_X\xrightarrow{f} X$ sous le groupe $S_X$.
Soit $H\xrightarrow{\psi}G$ le groupe minimal compatible au $S_X$-torseur $\CT_X$.
Alors, pour tout  \'el\'ement de $n$-torsion $\alpha\in \Br(X)$ et tout $\sigma\in H^1(k,S_X)$,
 on a $f_{\sigma}^*(\alpha)\in \Br_H(\CT_{X,\sigma})$, o\`u $f_{\sigma}^*: \Br(X)\to \Br(\CT_{X,\sigma})$ est l'homomorphisme induit par $f_{\sigma}: \CT_{X,\sigma}\to X$.
\end{prop}

\begin{proof}
On peut supposer que $\sigma=0\in H^1(k,S_X)$.

Notons $\rho_H: H\times \CT_X\to \CT_X$ l'action de $H$ 
et $p_{1,H}: H\times \CT_X\to H$, $p_{2,H}: H\times \CT_X\to \CT_X$ les deux projections.
Soit $\CT_G$ un torseur universel de $n$-torsion pour $G$ sous le groupe $S_G$.

Apppliquant le corollaire \ref{BiDlem2.3.2} \`a $(G,X)$, 
on obtient :
pour tout $\alpha_1\in H^2(X, \mu_n)$, il existe un $\phi \in \Hom(S_G,S_X^*)$ et un $\beta\in H^2(G,\mu_n)$ 
tels que $(\rho^*-p_2^*)(\alpha_1)=\varepsilon (\phi)+p_1^*(\beta)$.

Puisque $f^*([\CT_X])=0\in H^1(\CT_X,S_X)$,  on a
$$(\psi\times f)^*(\varepsilon (\phi))=(\psi\times f)^* (\phi_*([\CT_G])\cup [\CT_X] )=\phi_*(\psi^*([\CT_G]))\cup f^*([\CT_X])=0.$$
Alors $(\rho_H^*-p_{2,H}^*)(f^*(\alpha_1) )= (\psi\times f)^*((\rho^*-p_2^*)(\alpha_1))=(\psi\times f)^*(p_1^*(\beta))=p_{1,H}^*(\psi^*(\beta))$.

D'apr\`es la suite exacte de Kummer, $(\rho_H^*-p_{2,H}^*)(f^*(\alpha))\sbt p_{1,H}^*\Br(H)$, d'o\`u le r\'esultat.
\end{proof}

\section{Rappel sur le sous-groupe de Brauer invariant}\label{BiD4}

Dans toute cette section,  $k$ est un corps quelconque de caract\'eristique $0$. 
Sauf  mention explicite du contraire,  une vari\'et\'e est une $k$-vari\'et\'e.

Dans cette section, on rappelle des notions et des r\'esultats dans \cite[\S 3]{C1} sur le sous-groupe de Brauer invariant.

Pour la d\'efinition du sous-groupe de Brauer invariant on renvoie le lecteur \`a la d\'efinition \ref{BiDdefBrainv}.

Soit $\mathbf{AB}$ la cat\'egorie des groupes ab\'eliens. 
Soit $\mathbf{GX}$ la cat\'egorie  des couples $(G,X)$ avec $G$ un groupe alg\'ebrique connexe et $X$ une $G$-vari\'et\'e lisse,
 et un morphisme $(H,Y)\to (G,X)$ dans $\mathbf{GX}$ est un couple $(\psi,f)$ avec $\psi: H\to G$ un homomorphisme et $f: Y\to X$ un $H$-morphisme,
  o\`u l'action de $H$ sur $X$ est induite par $\psi$.
Par d\'efinition, 
$$\Br_{-}(-): \mathbf{GX}\to \mathbf{AB}: (G,X)\mapsto \Br_G(X)$$
 est un foncteur contravariant.

\begin{exam}\label{BiDexam4.1}
(1) (\cite[Lem. 3.6]{C1}) Soit $G$ un groupe lin\'eaire connexe. Alors $\Br_G(G)=\Br_1(G)$.

(2) (\cite[Prop. 3.9 (3)]{C1}) Soient $G$ un groupe lin\'eaire connexe, $G_0\sbt G$ un sous-groupe ferm\'e connexe et $X:=G/G_0$.
Alors  $$\Br_G(X)=\Br_1(X,G):=\ker(\Br(X)\to \Br(G_{\bk})).$$
Le groupe $\Br_1(X,G)$ est d\'efini par Borovoi et Demarche dans \cite{BD} pour \'etudier l'approximation forte de $X$.

(3) Soit $A$ une vari\'et\'e ab\'elienne. 
L'auteur ne sait pas identifier le groupe $\Br_A(A)$.
Par exemple, l'auteur ne sait pas si $\Br_A(A)\sbt \Br_1(A)$ ou si $\Br_A(A)\supset \Br_1(A)$.
\end{exam}

\medskip

Au vu de l'exemple  \ref{BiDexam4.1} (3), dans la suite du pr\'esent article, on suppose que $G$ est un groupe lin\'eaire.

Soient $G$ un groupe lin\'eaire connexe et $X$ une $G$-vari\'et\'e lisse g\'eom\'e\-triquement int\`egre.
Notons $\rho: G\times X\to X$ l'action et $p_1: G\times X\to G$, $p_2:G\times X\to X$ les deux projections.

(1) Puisque $\Br_1(G\times X)\cong \Br_e(G)\oplus \Br_1(X)$ (\cite[Lem. 6.6]{S}), on peut obtenir facilement (\cite[Prop. 3.2 (4)]{C1}):
\begin{equation}\label{BiDinvequa2.1}
\Br_1(X)\sbt \Br_G(X).
\end{equation}

(2)  Puisque $p_1^*|_{\Br_e(G)}: \Br_e(G)\to \Br(G\times X) $ est injectif, par la d\'efinition de $\Br_G(X)$,
 il existe un unique homomorphisme $\Br_G(X)\xrightarrow{\lambda}\Br_e(G) $ tel que (\cite[(3.4)]{C1})
$$p_1^* \circ \lambda=\rho^*-p_2^*: \Br_G(X)\to \Br(G\times X).$$
Le $\lambda: \Br_G(X)\to \Br_e(G)$ est appel\'e \emph{l'homomorphisme de Sansuc} (\cite[D\'ef. 3.8]{C1}).

(3)  Pour toute extension de corps $K/k$, et tous $x\in X(K)$, $g\in G(K)$, $\alpha\in \Br_G(X)$, on a  (\cite[Prop. 3.9 (1)]{C1}):
$$(g\cdot x)^*(\alpha )= g^*(\lambda (\alpha))+x^*(\alpha)\in \Br(K).$$
Alors, dans le cas o\`u $k$ est un corps de nombres, on a:
\begin{equation}\label{BiDinvequa2.3}
G(\RA_k)^{\Br_a(G)}\cdot X(\RA_k)^{\Br_G(X)}=X(\RA_k)^{\Br_G(X)}.
\end{equation}

La formule (\ref{BiDinvequa2.1}) et \cite[Cor. 3.6]{HS13} impliquent directement:

\begin{cor}\label{BiD4cor1}
Soit $(G,X)\in \mathbf{GX}$ un objet.
Si $X(\RA_k)^{\Br_G(X)}\neq\emptyset$, 
alors, pour tout entier $n \geq 2$, il existe un un torseur universel de $n$-torsion pour $X$.
\end{cor}

Pour un torseur sous un groupe lin\'eaire connexe, Sansuc a construit une suite exacte dans \cite[Prop. 6.10]{S}, qui est appel\'ee \emph{la suite exacte de Sansuc}.
La proposition suivante dit que les sous-groupes de Brauer invariant sont compatibles avec la suite exacte de Sansuc.

\begin{prop}\label{BiDRprop3}(\cite[Cor. 3.11(2)]{C1})
Soit $1\to N\to H\xrightarrow{\psi}G\to 1$ une suite exacte de groupes lin\'eaires connexes.
Soit $(\psi, f): (H,Y)\to (G,X)$ un morphisme dans $\mathbf{GX}$ tel que $X$ soit g\'eom\'e\-triquement int\`egre sur $k$ et $Y\to X $ soit un $N$-torseur, o\`u l'action de $N$ sur $Y$ est induite par celle de $H$.
Alors $f^*: \Br(X)\to \Br(Y)$ satisfait $(f^*)^{-1}\Br_H(Y)=\Br_G(X)$ 
et on a  une suite exacte, fonctorielle en $(X,Y,f,N)$:
$$\Pic(Y)\to  \Pic(N)\to \Br_G(X)\xrightarrow{f^*} \Br_H(Y)\xrightarrow{\lambda}\Br_e(N),$$
o\`u $\lambda: \Br_H(Y)\sbt \Br_N(Y)\to \Br_e(N)$ est l'homomorphisme de Sansuc.
\end{prop}

Pour une fibration $f:X\to T$ et tout $t\in T(k)$, on note $i_t: X_t\sbt X$ la fibre et
on a la sp\'ecialisation du groupe de Brauer $i_t^*: \Br(X)\to \Br(X_t)$.
La proposition suivante dit que, si la fibration $f$ est compatible avec des actions des groupes lin\'eaires, 
alors les sous-groupes de Brauer invariants sont compatibles avec la sp\'ecialisation du groupe de Brauer.

\begin{prop}\label{BiDRprop4}(\cite[Prop. 3.13]{C1})
Soit $1\to G_0\xrightarrow{\phi} G\xrightarrow{\psi}T\to 1$ une suite exacte de groupes lin\'eaires connexes avec $T$ un tore.
Soient $X$ une $G$-vari\'et\'e lisse g\'eom\'e\-triquement int\`egre et $X\xrightarrow{f}T$ un $G$-morphisme.
Notons $\Br_1(G)\xrightarrow{\phi_*}\Br_1(G_0)$ l'homomorphisme induit par $\phi$.
Alors, pour tout $t\in T(k)$, on a

(1)  la fibre $i_t: X_t\sbt X$ est $G_0$-invariante;

(2) on a une suite exacte naturelle
$$\Br_e(T)\to \Br_G(X)\xrightarrow{i_t^*} \Br_{G_0}(X_t)\to \coker (\phi_*);$$

(3) (\cite[Lem. 5.5]{C1}) si $k$ est un corps de nombres et $H^3(k,T^*)=0$, on a $\coker(\phi_*)=0$ et $i_t^*$ est surjectif.
\end{prop}

\section{La descente par rapport au sous-groupe de Brauer invariant}

Dans toute cette section,  $k$ est un corps de nombres. 
Sauf  mention explicite du contraire, une vari\'et\'e est  une $k$-vari\'et\'e.

La m\'ethode de descente des points ad\'eliques est \'etablie par Colliot-Th\'el\`ene et Sansuc dans \cite{CTS}.
Dans \cite{C1}, l'auteur \'etudie la m\'ethode de descente des points ad\'eliques orthogonaux aux sous-groupes de Brauer invariants et \'etablit le r\'esultat:
pour  un groupe lin\'eaire connexe $G$, une vari\'et\'e lisse g\'eom\'e\-triquement int\`egre $Z$ et  un $G$-torseur $p: X\to Z$, 

(1) on a (\cite[Thm. 5.9]{C1}):
\begin{equation}\label{BiDprop4.2e}
Z(\RA_k)^{\Br(Z)}=\cup_{\sigma\in H^1(k,G)}p_{\sigma}(X_{\sigma}(\RA_k)^{\Br_{G_{\sigma}}(X_{\sigma})});
\end{equation}

(2) si $G$ est un tore quasi-trivial, on a (\cite[Prop. 5.2]{C1})
\begin{equation}\label{BiDprop4.1e}
Z(\RA_k)^{(p^*)^{-1}B}=p(X(\RA_k)^{B}),
\end{equation}
pour tout sous-groupe $B\sbt \Br_G(X)$, o\`u $p^*: \Br(Z)\to \Br(X)$;

(3) pour tout homomorphisme surjectif $\psi: H\to G$ de groupes lin\'eaires connexes, on a (\cite[Thm. 5.1]{CLX}):
\begin{equation}\label{BiDprop4.1e2}
G(\RA_k)^{\Br_1(G)}=\psi(H(\RA_k)^{\Br_1(H)})\cdot G(k).
\end{equation}

La proposition \ref{BiDprop4.1} et la proposition \ref{BiDprop4.2} suivantes sont quelques variantes de ce r\'esultat.
Plus pr\'ecis\'ement, la proposition \ref{BiDprop4.1} est une variante de (2) pour $G$ un groupe fini commutatif et sa d\'emonstration utilise (2) et (3) mais pas (1).
La proposition \ref{BiDprop4.2} est une variante de (1) en rempla\c{c}ant ``$\Br$'' par ``$\et,\Br$'' et en rempla\c{c}ant ``$\Br_G$'' par ``$G-\et,\Br_G$'',
donc elle est une version limite de (1) pour tout $k$-torseur $Z'\to Z$ sous un $k$-groupe fini, et sa d\'emonstration utilise (1) mais pas (2) et (3).

\begin{prop}\label{BiDprop4.1}
Soient $G$, $H$ deux groupes lin\'eaires connexes et $\psi: H\to G$ un homomorphisme surjectif de noyau  $S$ fini.
Soient $X$ (resp. $Y$) une $G$-vari\'et\'e (resp. $H$-vari\'et\'e) lisse g\'eom\'etriquement int\`egre et $f: Y\to X$ un $H$-morphisme tels que $Y$ soit un $S$-torseur sur $X$, o\`u l'action de $S$ est induite par l'action de $H$.
Alors, pour tout $\sigma\in H^1(k,S)$, le tordu $Y_{\sigma}$ est une $H$-vari\'et\'e et on a:
$$X(\RA_k)^{\Br_G(X)}=\cup_{\sigma\in H^1(k,S)}f_{\sigma}(Y_{\sigma}(\RA_k)^{\Br_H(Y_{\sigma})}).$$
\end{prop}

\begin{proof}
On construira les diagrammes ci-dessous
$$\xymatrix{0\ar[r]&S\ar[r]\ar[d]\ar@{}[rd]|{\lrcorner}&  T_0\ar[r]^{\varphi}\ar[d]& T\ar[d]^=\ar[r]   &0 &&&\\
&H\ar[r]\ar[rd]_{\psi} &H_0\ar[r]\ar[d]^{\psi_0} &T &\ar@{}[d]|{\text{et}} &Y\ar[r]^i\ar[rd]_f &Y_0\ar[r]^{\phi}\ar[d]^{f_0} &T\\ &&G&&& &X& } $$
o\`u le diagramme \`a gauche est un diagramme 
de groupes alg\'ebriques, le diagramme \`a droite est un diagramme de vari\'et\'es lisses, 
 chaque groupe dans le diagramme \`a gauche agit sur la vari\'et\'e dans le diagramme \`a droite avec le m\^eme position et ces actions sont compatibles avec tous les morphismes.

\medskip

Puisque $H$ est connexe, $S$ est contenu dans le centre de $H$. Donc $S$ est commutatif.
Une r\'esolution coflasque (\cite[Prop. 1.3]{CTS1}) induit une suite exacte
\begin{equation}\label{BiDprop4.1e1}
0\to S\to T_0\xrightarrow{\varphi} T\to 0
\end{equation}
o\`u  $T_0$  est un tore quasi-trivial et $T$ est un tore coflasque, i.e. $H^1(k',T^*)=0$ pour toute extension $k'/k$.
 Puisque $H^3(k,T^*) \cong \prod_{v\in \infty_k}H^3(k_v,T^*) \cong  \prod_{v\in \infty_k}H^1(k_v,T^*)$ (\cite[Lem. 5.4]{C1}), on a $H^3(k,T^*)=0$.

Soit $H_0:=H\times^ST_0$. 
Alors $H_0$ est un groupe lin\'eaire connexe et $H \xrightarrow{\psi} G$ induit une suite exacte 
$$1\to T_0\to H_0\xrightarrow{\psi_0} G\to 1.$$
Soit $Y_0:=Y\times^ST_0$. Notons $i: Y\to Y_0$ l'immersion ferm\'ee canonique. 
Alors $ Y_0$ est une $H_0$-vari\'et\'e et $f$ induit un $H_0$-morphisme $Y_0\xrightarrow{f_0} X$ 
tels que $f_0$ est un $T_0$-torseur.
D'apr\`es (\ref{BiDprop4.1e}) et la proposition \ref{BiDRprop3}, on a 
$$X(\RA_k)^{\Br_G(X)}=f_0(Y_0(\RA_k)^{\Br_{H_0}(Y_0)}).$$

L'isomorphisme $Y_0\times^{T_0}T\cong Y\times^ST_0\times^{T_0}T\cong X\times T$ 
induit un $T_0$-morphisme $\phi: Y_0\to T$ tel que $\phi^{-1}(e_T)=i(Y)$.
D'apr\`es des arguments classiques (voir la d\'emonstration de \cite[Thm. 5.9]{C1}), 
pour tout $t\in T(k)$, on a $\phi^{-1}(t)\cong Y_{\partial(t)}$ 
et le morphisme $\phi^{-1}(t)\hookrightarrow Y_0\xrightarrow{f_0}X $ est exactement $f_{\partial(t)}$, 
o\`u $\partial: T(k)\to H^1(k,S)$ est l'homomorphisme induit par (\ref{BiDprop4.1e1}).
Puisque $H^3(k,T^*)=0$, d'apr\`es la proposition \ref{BiDRprop4}, $\phi^{-1}(t)$ est une $H$-vari\'et\'e 
et l'homomorphisme canonique $\Br_{H_0}(Y_0)\to \Br_H(\phi^{-1}(t))$ est surjectif pour tout $t\in T(k)$.

D'apr\`es (\ref{BiDinvequa2.3}), $Y_0(\RA_k)^{\Br_{H_0}(Y_0)}$ est $T_0(\RA_k)^{\Br_1(T_0)}$-invariant.
On applique (\ref{BiDprop4.1e2}) \`a $\varphi$ et on a 
$$T(\RA_k)^{\Br_1(T)}=\varphi(T_0(\RA_k)^{\Br_1(T_0)})\cdot T(k).$$
Puisque $\phi(Y_0(\RA_k)^{\Br_{H_0}(Y_0)})\sbt T(\RA_k)^{\Br_1(T)}$, on a:
$$Y_0(\RA_k)^{\Br_{H_0}(Y_0)}=T_0(\RA_k)^{\Br_1(T_0)}\cdot (\sqcup_{t\in T(k)}\phi^{-1}(t)(\RA_k)^{\Br_H(\phi^{-1}(t))}), $$
et donc $X(\RA_k)^{\Br_G(X)}=f_0 [\sqcup_{t\in T(k)}\phi^{-1}(t)(\RA_k)^{\Br_H(\phi^{-1}(t))}]
=\cup_{t\in T(k)}f_{\partial(t)}[Y_{\partial(t)}(\RA_k)^{\Br_H(Y_{\partial(t)})}] .$
\end{proof}

Rappelons la d\'efinition de $X(\RA_k)^{G-\et,\Br_G}$ dans (\ref{BiDdef1e1}).

Pour toute vari\'et\'e lisse $X$, d\'efinissons $X(\RA_k^{nc})$ l'espace des points ad\'eliques de $X$ hors des places complexes,
i.e. on a 
$X(\RA_k)\cong (\prod_{v\ \text{complexe}}X(k_v))\times X(\RA_k^{nc}).$
De plus, on a:
\begin{equation}\label{BiDlem4.2.1-e}
X(\RA_k)^{ob}\cong (\prod_{v\ \text{complexe}}X(k_v))\times X(\RA_k^{nc})^{ob}
\end{equation}
pour l'obstruction $ob=\Br(X) $ ou  $ob=\Br_1(X)$ ou $ob=\et, \Br$ ou, si $X$ est une $G$-vari\'et\'e pour un groupe lin\'eaire connexe $G$, pour $ob=\Br_G(X)$ ou $ob=G-\et,\Br_G$.

Le lemme suivant est bien connu (voir \cite[Lem. 2.2.8]{Dth} pour une variante).

\begin{lem}\label{BiDlem4.2.1}
Soient $X$ une vari\'et\'e lisse et $\{X_i\}_{i\in I}$ les composantes connexes de $X$ 
telles que $X_i$ soit g\'eom\'etriquement int\`egre pour tout $i\in I$.
Alors on a:
$$X(\RA_k^{nc})^{\Br_1(X)}=\coprod_{i\in I}  X_i(\RA_k^{nc})^{\Br_1(X_i)},\ \ \  
X(\RA_k^{nc})^{\et, \Br}=\coprod_{i\in I}  X_i(\RA_k^{nc})^{\et,\Br}$$
et, si $X$ est une $G$-vari\'et\'e pour un groupe lin\'eaire connexe $G$, on a:
$$X(\RA_k^{nc})^{\Br_G(X)}=\coprod_{i\in I}  X_i(\RA_k^{nc})^{\Br_G(X_i)} \ \ \ \text{et} \ \ \  
X(\RA_k^{nc})^{G-\et,\Br_G}=\coprod_{i\in I}  X_i(\RA_k^{nc})^{G-\et, \Br_G} .$$
\end{lem}

\begin{proof}
Puisque le groupe de Brauer (resp. le sous-groupe de Brauer $G$-invariant, resp. l'ensemble des $F$-torseurs, resp. l'ensemble des $F$-torseurs $G$-compatibles pour un $k$-groupe fini $F$) de $X$
est la somme directe de celui des composantes connexes de $X$, on obtient l'inclusion $\supset $ dans les quatre cas ci-dessus.

Par ailleurs, soit $\pi_0(X)$ le sch\'ema des composantes connexes g\'eom\'etriques de $X$, 
i.e. $\pi_0(X)$ est un $k$-sch\'ema fini \'etale
et il existe un $k$-morphisme surjectif $\phi: X\to \pi_0(X)$ de fibres g\'eom\'e\-triquement int\`egres.
Pour tout $k$-sch\'ema $V$ fini \'etale connexe, $V(\RA_k^{nc})\neq\emptyset$ implique $V\cong Spec\ k$.
D'apr\`es \cite[Prop. 3.3]{LX} (un r\'esultat inspir\'e par Stoll), on a $\pi_0(X)(\RA_k^{nc})^{\Br(\pi_0(X))}=\pi_0(X)(k)$.
Par d\'efinition, $\phi^*(\Br(\pi_0(X)))\sbt \Br_1(X)$, $\phi^*(\Br(\pi_0(X)))\sbt \Br_G(X)$ et donc on obtient l'inclusion $\sbt$.
\end{proof}

Les deux lemmes suivants sont bien connus.

\begin{lem}\label{BiDlem4.2.2}
Soient $X$ une vari\'et\'e lisse, $L$ un groupe lin\'eaire quelconque et $h: V\to X$ un $L$-torseur.
Alors, pour tout $x\in X(\RA_k)$, l'ensemble $\{\sigma\in H^1(k,L):\ x\in h_{\sigma}(V_{\sigma}(\RA_k))\}$ est fini.
\end{lem}

\begin{proof}
Le r\'esultat d\'ecoule du fait que, pour tout $\sigma\in H^1(k,L)$, l'ensemble $\Sha^1(k,L_{\sigma})$ est fini (\cite[\S III.4.6]{Se}).
\end{proof}

 Voyons \cite[Prop. 5.3.2]{sko} et \cite[Lem. 6.3]{CDX} pour des r\'esultats similaires.

\begin{lem}(M. Stoll \cite{St}, cf. \cite[Lem. 7.1]{CDX}) \label{BiDlem4.2.3}
Soit $X$ une vari\'et\'e lisse g\'eom\'e\-triquement int\`egre, $F$ un $k$-groupe fini et $f: Y\to X$ un $F$-torseur.
Supposons qu'il existe un $x\in X(\RA_k)^{\et,\Br}$.
Alors il existe un $\sigma\in H^1(k,F)$, un sous-groupe ferm\'e $F'\sbt F_{\sigma}$, une composante connexe $Y'\sbt Y_{\sigma}$ 
tels que $Y'$ soit g\'eom\'e\-triquement int\`egre et  $F'$-invariant, $f':=f_{\sigma}|_{Y'}: Y'\to X$ soit un $F'$-torseur et que $x\in f'(Y'(\RA_k)^{\Br(Y')})$.
\end{lem}

La proposition suivante est une \'etape interm\'ediaire importante dans la d\'emonstration du th\'eor\`eme \ref{BiDthm2}.

\begin{prop}\label{BiDprop4.2}
Soient $G$ un groupe lin\'eaire connexe, $Z$ une vari\'et\'e lisse g\'eom\'e\-triquement int\`egre et $p: X\to Z$ un $G$-torseur.
Alors: 
$$Z(\RA_k)^{\et, \Br}=\cup_{\sigma\in H^1(k,G)}p_{\sigma}(X_{\sigma}(\RA_k)^{G_\sigma-\et,\Br_{G_{\sigma}}}).$$
\end{prop}

\begin{proof}
L'inclusion $\supset$ d\'ecoule du fait que,
pour tout torseur $V\to Z$ sous un $k$-groupe fini, l'image r\'eciproque $X\times_ZV\to X$ est $G$-compatible.

Pour l'inclusion $\sbt$, on peut supposer que $Z(\RA_k)^{\et, \Br}\neq\emptyset$.

On fixe un point $z\in Z(\RA_k)^{\et, \Br}$.

Soit $\Delta$ l'ensemble des $\sigma\in H^1(k,G)$ tels que
$p_{\sigma}^{-1}(z)\cap X_{\sigma}(\RA_k)^{\Br_{G_{\sigma}}(X_{\sigma})}\neq \emptyset. $
Alors $\Delta\neq\emptyset$ par (\ref{BiDprop4.2e}).
Pour tout $\sigma\in \Delta$, on fixe un point 
$x_{\sigma}\in p_{\sigma}^{-1}(z)\cap X_{\sigma}(\RA_k)^{\Br_{G_{\sigma}}(X_{\sigma})}$.
Ceci induit un isomorphisme: 
$$\Psi_{\sigma}: G(\RA_k)\to p_{\sigma}^{-1}(z)(\RA_k): g\mapsto g\cdot x_{\sigma}.$$
Notons:
$$E_{0,\sigma}:=\Psi_{\sigma}^{-1}(p_{\sigma}^{-1}(z)\cap X_{\sigma}(\RA_k)^{\Br_{G_{\sigma}}(X_{\sigma})})\ \ \ 
\text{et}\ \ \ E_0:=\sqcup_{\sigma\in \Delta}E_{0,\sigma}.$$

Pour tout $\sigma\in \Delta$, soit $G_{\sigma}(\RA_k)\xrightarrow{a_{\sigma}}\Hom(\Br_a(G_{\sigma}),\BQ/\BZ)$ 
l'homomorphisme induit par l'accouplement de Brauer-Manin.
Donc $\Ker(a_{\sigma})=G_{\sigma}(\RA_k)^{\Br_a(G_{\sigma})}$.
Notons 
$$K_{a,\Delta}:=\prod_{\sigma\in \Delta} \Ker(a_{\sigma}) \ \ \ \text{et} \ \ \ 
G_{\Delta}(\RA_k):=\sqcup_{\sigma\in \Delta}G_{\sigma}(\RA_k).$$
D\'efinissons l'action de $\Ker(a_{\sigma})$ sur $G_{\sigma}(\RA_k)$ par la multiplication  \`a gauche.
Ceci induit une unique action de $K_{a,\Delta}$ sur $G_{\Delta}(\RA_k)$
 telle que l'action de $\Ker(\sigma_1)$ sur $G_{\sigma_2}(\RA_k)$ soit l'identit\'e pour tous $\sigma_1\neq \sigma_2$.
D'apr\`es (\ref{BiDinvequa2.3}), $E_0$ est $K_{a,\Delta}$-invariant.

Soit $\CS$ l'ensemble des couples $(F, V\xrightarrow{f} Z)$ avec $F$ un $k$-groupe fini 
et $V\xrightarrow{f} Z $ un $F$-torseur tel que $V$ soit g\'eom\'e\-triquement int\`egre.
On d\'efinit un ordre partiel:
pour tous $(F_1,V_1), (F_2,V_2) \in \CS$, 
on a $(F_1,V_1)\leq (F_2,V_2) $ si et seulement s'il existe un $\sigma\in H^1(k,F_1)$ 
et un homomorphisme surjectif $\phi: F_2\to F_{1,\sigma}$ tels que $\phi_*([V_2])=[V_{1,\sigma}]$.

Pour tout $(\delta,\sigma)\in H^1(k,F)\times H^1(k,G)$, soit $Y_{\sigma,\delta}:=X_{\sigma}\times_ZV_{\delta}$.
On a un diagramme commutatif de $F_{\delta}\times G_{\sigma}$-vari\'et\'es et de $F_{\delta}\times G_{\sigma}$-morphismes:
$$\xymatrixcolsep{7pc}\xymatrix{Y_{\sigma,\delta}\ar@{}[rd]|{\square}\ar[r]^{f^{\sigma}_{\delta}} \ar[d]^{p^{\delta}_{\sigma}} &X_{\sigma}\ar[d]^{p_{\sigma}}\\
V_{\delta}\ar[r]^{f_{\delta}}&Z,
}$$
tel que toute verticale soit un $G_{\sigma}$-torseur et que toute horizontale soit un $F_{\delta}$-torseur.

Pour tout $(F, V\xrightarrow{f} Z)\in \CS $ et tout $\sigma\in \Delta$,
notons 
\begin{equation}\label{BiDlem4.2e}
E_{F,V,\sigma}:=\Psi_{\sigma}^{-1}(p_{\sigma}^{-1}(z)\cap [\cup_{\delta\in H^1(k,F)}  f^{\sigma}_{\delta} (Y_{\sigma,\delta}(\RA_k)^{\Br_{G_{\sigma}}(Y_{\sigma,\delta}) } )] )\sbt G_{\sigma}(\RA_k)
\end{equation}
et $E_{F,V}:=\sqcup_{\sigma\in \Delta}E_{F,V,\sigma}\sbt G_{\Delta}(\RA_k)$.

\begin{lem}\label{BiDlem4.2}
Pour tout $(F, V\xrightarrow{f} Z)\in \CS $, on a:

(1) l'ensemble $E_{F,V}$ est un sous-ensemble non vide ferm\'e $K_{a,\Delta}$-invariant de $E_0$;

(2)  pour tout $(F_1, V_1)\in \CS $ v\'erifiant $(F, V)\leq  (F_1, V_1)$,  on a $E_{F_1,V_1}\sbt E_{F,V}$;

(3) l'ensemble  $\CS$ est un ensemble ordonn\'e filtrant.
\end{lem}

\begin{proof}
Pour tout $\delta\in H^1(k,F)$ et tout $\sigma\in H^1(k,G)$, le morphisme $f^{\sigma}_{\delta}$ est fini.
D'apr\`es (\ref{BiDinvequa2.3}), 
$Y_{\sigma,\delta}(\RA_k)^{\Br_{G_{\sigma}}(Y_{\sigma,\delta})}$ est $\Ker(a_{\sigma})$-invariant 
et
$$f^{\sigma}_{\delta}(Y_{\sigma,\delta}(\RA_k)^{\Br_{G_{\sigma}}(Y_{\sigma,\delta})})
\sbt X_{\sigma}(\RA_k)^{\Br_{G_{\sigma}}(X_{\sigma})}$$
 est ferm\'e (\cite[Prop. 4.4]{Co}) et $\Ker(a_{\sigma})$-invariant.
 Ainsi
 $$\Psi_{\sigma}^{-1}[p_{\sigma}^{-1}(z)\cap 
  f^{\sigma}_{\delta}(Y_{\sigma,\delta}(\RA_k)^{\Br_{G_{\sigma}} (Y_{\sigma,\delta})})]
  \sbt E_{0,\sigma}$$
 est ferm\'e et $\Ker(a_{\sigma})$-invariant.

 Appliquant (\ref{BiDprop4.2e})  et le lemme \ref{BiDlem4.2.2} \`a 
$Y_{\sigma, \delta}\xrightarrow{p^{\delta}_{\sigma}} V_{\delta}$,
 il existe au moins un et au plus un nombre fini de $(\delta,\sigma)\in H^1(k,F)\times H^1(k,G)$ tels que 
 $$(f_{\delta}\circ  p_{\sigma}^{\delta})^{-1}(z)\cap Y_{\sigma, \delta}(\RA_k)^{\Br_{G_{\sigma}}(Y_{\sigma,\delta})}\neq\emptyset.  $$
Alors $p_{\sigma}^{-1}(z)\cap X_{\sigma}(\RA_k)^{\Br_{G_{\sigma}}(X_{\sigma})}\neq \emptyset $ 
et donc un tel $\sigma$ est dans $\Delta$.
Alors $E_{F,V}\neq\emptyset$ et (1) d\'ecoule du premier paragraphe.

L'\'enonc\'e (2) d\'ecoule de la  fonctorialit\'e de l'accouplement de Brauer-Manin.

Pour tous $(F_1,V_1), (F_2,V_2) \in \CS$, on a un $F_1\times F_2$-torseur $V_1\times_ZV_2\to Z$.
Par hypoth\`ese,  il existe un $(\sigma_1,\sigma_2)\in H^1(k,F_1)\times H^1(k,F_2)$ 
tel que $(V_{1,\sigma_1}\times_ZV_{2,\sigma_2})(\RA_k)^{\Br(V_{1,\sigma_1}\times_ZV_{2,\sigma_2})}\neq\emptyset$.
D'apr\`es le lemme \ref{BiDlem4.2.1} et (\ref{BiDlem4.2.1-e}), il existe un $k$-sous-groupe ferm\'e $F_3\sbt F_{1,\sigma_1}\times F_{2,\sigma_2}$ 
et une composante connexe $V_3\sbt V_{1,\sigma_1}\times_ZV_{2,\sigma_2}$
tels que $V_3$ soit g\'eom\'e\-triquement int\`egre et que
$V_3\to Z$ soit un $F_3$-torseur compatible avec l'action de $F_{1,\sigma_1}\times F_{2,\sigma_2}$ sur $V_{1,\sigma_1}\times_ZV_{2,\sigma_2}$.
Alors le morphisme $h_1: V_3\sbt V_{1,\sigma_1}\times_ZV_{2,\sigma_2}\to V_{1,\sigma_1}$ 
est compatible avec $\phi_1: F_3\sbt F_{1,\sigma_1}\times F_{2,\sigma_2}\to F_{1,\sigma_1}$.
Puisque $V_{1,\sigma_1}$ est g\'eom\'e\-triquement int\`egre, le morphisme $h_1$ est surjectif et donc $\phi_1$ est surjectif.
Alors $[V_{1,\sigma_1}]=\phi_{1,*}([V_3])$ et $(F_1,V_1)\leq (F_3,V_3)$.
Par ailleurs, $(F_2,V_2)\leq (F_3,V_3)$, d'o\`u l'\'enonc\'e (3).
\end{proof}

Soient $\CB:=\sqcup_{\sigma\in \Delta}\Hom(\Br_a(G_{\sigma}),\BQ/\BZ) $ et
$$a_{\Delta}: G_{\Delta}=\sqcup_{\sigma\in \Delta}G_{\sigma}(\RA_k)\xrightarrow{\sqcup_{\sigma\in \Delta}a_{\sigma}} \sqcup_{\sigma\in \Delta}\Hom(\Br_a(G_{\sigma}),\BQ/\BZ)=\CB.$$
En tant qu'ensembles, 
on a $\Im(a_{\Delta})\cong K_{a,\Delta}\backslash G_{\Delta}$.
 L'espace $\Hom(\Br_a(G_{\sigma}),\BQ/\BZ)$ est compact, car $\Br_a(G_{\sigma})$ est discret.
D'apr\`es  le lemme \ref{BiDlem4.2.2}, $\Delta$ est fini et donc $\CB$ est compact.
Puisque $a_{\sigma}$ est continu et ouvert (\cite[Lem. 4.1]{C1}), l'application $a_{\Delta}$ est ouverte.
Donc l'image d'un sous-ensemble ferm\'e $K_{a,\Delta}$-invariant est ferm\'ee.
Alors $a_{\Delta}(E_{F,V})\sbt \CB$ est ferm\'e non vide pour tout $(F, V)\in \CS$.
Puisque $\CB $ est compact et que $\CS$ est un ensemble ordonn\'e filtrant, d'apr\`es le lemme \ref{BiDlem4.2} (2),
l'intersection 
$$\bigcap_{(F,V)\in \CS}a_{\Delta}(E_{F,V})\neq\emptyset\ \ \   \text{et donc}\ \ \  E_{\infty}:=\bigcap_{(F,V)\in \CS}E_{F,V}\neq\emptyset.$$
Il existe un $\sigma\in \Delta$ tel que $E_{\infty}\cap E_{\sigma}\neq\emptyset$.

Soient $g\in E_{\infty}\cap E_{\sigma}$ et $x:=\Psi_{\sigma}(g)=g\cdot x_{\sigma}$.
Alors $p_{\sigma}(x)=z$ et, d'apr\`es (\ref{BiDlem4.2e}), on a
$$x\in \bigcap_{(F,V)\in \CS}[\bigcup_{\delta\in H^1(k,F)}  f^{\sigma}_{\delta} (Y_{\sigma,\delta}(\RA_k)^{\Br_{G_{\sigma}}(Y_{\sigma,\delta}) } )].$$
D'apr\`es le corollaire \ref{BiDcor2.2},
 tout torseur $G$-compatible sous un $k$-groupe fini sur $X$ provient d'un torseur sur $Z$.
 D'apr\`es le lemme de Stoll (Lem. \ref{BiDlem4.2.3}), il suffit de consid\'erer les torseurs g\'eom\'etriquement int\`egres.
 Donc $x\in X_{\sigma}(\RA_k)^{G_\sigma-\et,\Br_{G_{\sigma}}}$, d'o\`u le r\'esultat.
\end{proof}

La proposition suivante est une g\'en\'eralisation de \cite[Rem. 7.5]{CDX}.

\begin{prop}\label{BiDprop4.3}
Soit $X$ une vari\'et\'e lisse g\'eom\'e\-triquement int\`egre.  Soit 
$$1\to N \to L \xrightarrow{\psi} F\to 1$$
une suite exacte de groupes lin\'eaires avec $F$ fini.
Soient $ V\to X$ un $L$-torseur et $Y:=V/N \to X$ le $F$-torseur induit par $\psi$, i.e. $[Y]=\psi_*([V])$.
Faisons  l'une ou l'autre des hypoth\`eses  :

(1) le groupe $N$ est connexe;

 (2) le groupe $L$ est fini et $N$ est contenu dans le centre de $L$.

Alors, pour tout $\sigma\in H^1(k,F)$ avec $Y_{\sigma}(\RA_k)^{\Br_1(Y_{\sigma})}\neq\emptyset$,
il existe un $\alpha\in H^1(k,L)$ tel que $\psi_*(\alpha)=\sigma$.
\end{prop}

\begin{proof}
Le cas o\`u $N$ est connexe est exactement \cite[Rem. 7.5]{CDX}.

On consid\`ere le cas (2). Dans ce cas, $N$ est un $k$-groupe fini commutatif.
La r\'esolution flasque (\cite[Prop. 1.3]{CTS1}) donne une suite exacte 
$0\to N\to T\to T_0\to 0$ avec $T$ un tore et $T_0$ un tore quasi-trivial.
Soit $L':=L\times^NT $. 
Alors $L'$ est un groupe lin\'eaire, car  $N$ est contenu dans le centre de $L$.
Ceci induit un diagramme commutatif de suites exactes et de colonnes exactes:

$$\xymatrix{&1\ar[d]&1\ar[d]&&\\
1\ar[r]&N\ar[r]\ar[d]&L\ar[r]\ar[d]^{\psi_2}&F\ar[d]^=\ar[r]&1\\
1\ar[r]&T\ar[r]\ar[d]&L'\ar[r]^{\psi_1}\ar[d]&F\ar[r]&1\\
&T_0\ar[r]^=\ar[d]&T_0\ar[d]&&\\
&0&0&&.
}$$
Appliquons le cas (1) au $L'$-torseur $\psi_{2,*}([V])$. On obtient un $\beta\in H^1(k,L')$ tel que $\psi_{1,*}(\beta)=\sigma$.
Puisque $H^1(k,T_0)=0$, il existe un $\alpha\in H^1(k,L)$ tel que $\psi_{2,*}(\alpha)=\beta$ et donc $\psi_*(\alpha)=\sigma$.
\end{proof}

La proposition \ref{BiDprop4.3} (2) et la formule (\ref{BiDinvequa2.1}) impliquent directement:

\begin{cor}\label{BiDprop4.3cor}
Sous les hypoth\`eses de la proposition \ref{BiDprop4.3} (2), soit $G$ un groupe lin\'eaire.
Pour tout $\sigma\in H^1(k,F)$, s'il existe une action de $G$ sur $Y_{\sigma}$ telle que $Y_{\sigma}(\RA_k)^{\Br_G(Y_{\sigma})}\neq\emptyset$, 
alors il existe un $\alpha\in H^1(k,L)$ tel que $\psi_*(\alpha)=\sigma$.
\end{cor}

\section{D\'emonstration du th\'eor\`eme \ref{BiDthm1}}\label{BiD6}

Dans toute cette section,  $k$ est un corps de nombres. 
Sauf  mention explicite du contraire, une vari\'et\'e est  une $k$-vari\'et\'e.

Dans toute cette section,  $G$  est un $k$-groupe lin\'eaire connexe et $(X,\rho)$ une $G$-vari\'et\'e lisse g\'eom\'e\-triquement int\`egre.

Pour tout $k$-groupe fini $F$ et tout $F$-torseur $f:Y\to X$, 
soit $(H_Y\xrightarrow{\psi_Y}G)$ le groupe minimal compatible au $F$-torseur $Y$ (cf. D\'efinition \ref{BiDdef2.1}).
Pour tout $\sigma\in H^1(k,F)$, le $F_{\sigma}$-torseur $f_{\sigma}: Y_{\sigma}\to X$ est $H_Y$-compatible, 
i.e. il existe une unique action de $H_Y$ sur $Y_{\sigma}$ telle que $f_{\sigma}$ soit un $H_Y$-morphisme.

Dans \S 1, on a d\'efini $X(\RA_k)^{\et, \Br}$ (cf. (\ref{BiDthm2e})) et $X(\RA_k)^{G-\et, \Br_G}$ (cf. (\ref{BiDdef1e1})).
On d\'efinit
$$X(\RA_k)^{\et, \Br_G}:=\bigcap_{\stackrel{f: Y\xrightarrow{F}X,}{ F\ \text{fini}}} \bigcup_{\sigma\in H^1(k,F)}f_{\sigma}(Y_{\sigma}(\RA_k)^{\Br_{H_Y}(Y_{\sigma})} ) $$
et
$$X(\RA_k)^{c.c.,\et, \Br_G}:=\bigcap_{\stackrel{f: Y\xrightarrow{F}X}{F\ \text{fini commutatif},\ Y\ \text{g\'eo. connexe} }} \bigcup_{\sigma\in H^1(k,F)}f_{\sigma}(Y_{\sigma}(\RA_k)^{\Br_{H_Y}(Y_{\sigma})})  $$
o\`u geo. connexe signifie g\'eom\'e\-triquement connexe et c.c. est une abr\'eviation de commutatif connexe. 
On a directement:
$$ X(\RA_k)^{\et,\Br_G} \sbt X(\RA_k)^{c.c.,\et, \Br_G}\ \ \ \text{et}\ \ \ 
X(\RA_k)^{\et,\Br}\sbt  X(\RA_k)^{\et,\Br_G} \sbt X(\RA_k)^{G-\et,\Br_G}.$$

\begin{prop}\label{BiDprop5.1}
On a $X(\RA_k)^{c.c.,\et, \Br_G}\sbt X(\RA_k)^{\Br(X)}$.
\end{prop}

\begin{proof}
Il suffit de montrer que, pour tout $\alpha\in \Br(X)$ et tout $x\in X(\RA_k)^{c.c.,\et, \Br_G}$,  on a $\alpha (x)=0$. 
On fixe un tel $x$ et un tel $\alpha$.

Il existe un entier $n$ tel que $n\cdot \alpha=0$.
D'apr\`es le corollaire \ref{BiD4cor1}, il existe un torseur universel de $n$-torsion $\CT_X\xrightarrow{f}X$ (un $S_X$-torseur).
Soit $H$ le groupe minimal compatible au $S_X$-torseur $\CT_X$.
Par hypoth\`ese, il existe un $\sigma\in H^1(k,S_X)$
 et un point ad\'elique $t\in \CT_{X,\sigma}(\RA_k)^{\Br_H(\CT_{X,\sigma})} $ tels que $f_{\sigma}(t)=x$.
 D'apr\`es la proposition \ref{BiDcor2.3.1}, $f_{\sigma}^*(\alpha)\in \Br_H(\CT_{X,\sigma})$.
Alors $\alpha(x)= f_{\sigma}^*(\alpha)(t)=0$.
\end{proof}

Le lemme suivant g\'en\'eralise un r\'esultat de Skorobogatov (\cite[Thm. 1.1]{Sk1}) et il g\'en\'eralise aussi \cite[Prop. 6.6]{CDX}.
Sa d\'emonstration suit l'id\'ee de \cite[p. 506]{Sk1} et de \cite[Prop. 5.17]{St}.

\begin{lem}\label{BiDlem5.2}
Soient $F$ un  $k$-groupe fini, $f: Y\to X$ un $F$ torseur
 et $(H_Y\xrightarrow{\psi_Y}G)$ le groupe minimal compatible au $F$-torseur $Y$.
  Supposons que $Y$ est g\'eom\'e\-triquement int\`egre.
 Alors 
 
 (1) on a  $X(\RA_k)^{\et, \Br_G}=\cup_{\sigma\in H^1(k,F)} f_{\sigma}[ Y_{\sigma}(\RA_k)^{\et, \Br_{H_Y}}];$
 
(2)  on a $X(\RA_k)^{\et, \Br}=\cup_{\sigma\in H^1(k,F)} f_{\sigma}[ Y_{\sigma}(\RA_k)^{\et, \Br}];$
 
 (3) si $\psi_Y: H_Y\iso G$ est un isomorphisme, on a
 $$X(\RA_k)^{G-\et, \Br_G}=\cup_{\sigma\in H^1(k,F)} f_{\sigma}[ Y_{\sigma}(\RA_k)^{G-\et, \Br_G}].$$
\end{lem}

\begin{proof}
L'inclusion $ \supset $ dans les trois cas est d\'efinie par le pullback des torseurs et la fonctorialit\'e de l'accouplement de Brauer-Manin.
On consid\`ere l'inclusion $\sbt$.

\medskip

Dans le cas (1), il suffit de montrer que, pour tout $x\in X(\RA_k)^{\et, \Br_G}$, 
il existe un $\sigma\in H^1(k,F)$ et un $y\in Y_{\sigma}(\RA_k)^{\et, \Br_{H_Y}}$ tels que $f_{\sigma}(y)=x$.
On fixe un tel $x$.

Pour tout $\sigma\in H^1(k,F)$, soient 
$$\Delta_{\sigma}:=f_{\sigma}^{-1}(x)\cap Y_{\sigma}(\RA_k),\ \ \ \Sigma:=\{\sigma\in H^1(k,F):\ \Delta_{\sigma}\neq\emptyset\}\ \ \ \text{et}\ \ \ \Delta:=\sqcup_{\sigma\in \Sigma}\Delta_{\sigma}.$$
D'apr\`es le lemme \ref{BiDlem4.2.2}, $\Delta$ et $ \Sigma$ sont finis.

 Soit $\CS$ l'ensemble des \emph{$X$-torseurs sur $Y$ sous $k$-groupes finis}
i.e. l'ensemble des quintuples $(\sigma, E,E\xrightarrow{\psi}F_{\sigma},V\xrightarrow{h_V} X, V\xrightarrow{h} Y_{\sigma})$ 
avec $\sigma\in H^1(k,F)$, $E$ un $k$-groupe fini, $\psi$ un homomorphisme surjectif, $V\xrightarrow{h_V} X$ un $E$-torseur et $h$ un $E$-morphisme sur $X$.
Alors $\psi_*([V])=[Y_{\sigma}]\in H^1(k,F)$ et $h: V\to Y_{\sigma}$ est un $\Ker(\psi)$-torseur. 
Donc $h_{\alpha}: V_{\alpha}\to Y_{\sigma+\psi_*(\alpha)}$ est un $\Ker(\psi_{\alpha})$-torseur pour tout $\alpha\in H^1(k,E)$.
Soit 
$$\Delta_V:=\{y\in \Delta:\ \exists \alpha\in H^1(k,E)\ \text{tel que}\ y\in h_{\alpha}(V_{\alpha}(\RA_k)^{\Br_{H_V}(V_{\alpha})})\} .$$
Par l'hypoth\`ese sur $x$, l'ensemble $\Delta_V$ est non vide.

On d\'efinit un ordre partiel de $\CS$:
pour tous $(\sigma_1, E_1,\psi_1,V_1,h_1), (\sigma_2, E_2,\psi_2,V_2,h_2) \in \CS$, 
on a $(\sigma_1, E_1,\psi_1,V_1,h_1)\leq (\sigma_2, E_2,\psi_2,V_2,h_2)$ si et seulement si $\sigma_1=\sigma_2$ et s'il existe un $\alpha \in H^1(k,E_1)$,
 un homomorphisme surjectif $\phi: E_2\to E_{1,\alpha}$ et un $E_2$-morphisme $h_{\phi}: V_2\to V_{1,\alpha}$ sur $Y_{\sigma_1}$.
 Dans ce cas, on a $\Delta_{V_2}\sbt \Delta_{V_1}$. 

Puisque $\Delta$ est fini, il existe un quintuple $(\sigma,E_0,\psi_0,V_0,h_0)$ dans $S$ tel que $\Delta_{V_0} $ soit minimal. On fixe un $y\in \Delta_{V_0}$.
Apr\`es avoir remplac\'e $\sigma$ par $\sigma+\psi_{0,*}(\alpha)$ pour certain $\alpha\in H^1(k,E_0)$, on peut supposer que $y\in Y_{\sigma}(\RA_k)$.

Pour tout torseur $Z\xrightarrow{f_1} Y_{\sigma}$ sous un $k$-groupe fini $F_1$, d'apr\`es \cite[Prop. 2.3 et (4)]{Sk1}, 
il existe un $(\sigma,E,\psi,h_V: V\to X,h)\in \CS$, un homomorphisme surjectif $\Ker(\psi)\to F_1$ et un $\Ker(\psi)$-morphisme $V\to Z$ sur $Y_{\sigma}$
avec 
\begin{equation}\label{BiDlem5.2e1}
 V:=R_{Y_{\sigma}\times_kF_{\sigma}/Y_{\sigma}}(Z\times_kF_{\sigma})\cong R_{Y_{\sigma}/X}(Z)\times_XY_{\sigma} \xrightarrow{h_V} X.
\end{equation}
Ceci induit
$$\Delta_V\sbt \cup_{\alpha\in H^1(k,\Ker(\psi))} h_{\alpha}(V_{\alpha}(\RA_k)^{\Br_{H_V}(V_{\alpha})})\sbt \cup_{\alpha\in H^1(k,F_1)} f_{1,\alpha}(Z_{\alpha}(\RA_k)^{\Br_{H_Z}(Z_{\alpha})}).$$
Par ailleurs, on a:
 $$(\sigma,E_0,\psi_0,V_0,h_0), (\sigma,E,\psi,V,h)\leq (\sigma, E_0\times_{F_{\sigma}} E,\psi_0\circ (id_{E_0}\times_{F_{\sigma}} \psi) , V_0\times_{Y_{\sigma}} V,h_0\circ (id_{V_0}\times_{Y_{\sigma}}h))$$ 
dans $\CS$, et donc $\Delta_{V_0}\supset \Delta_{V_0\times_{Y_{\sigma}} V}\sbt \Delta_V$.
Puisque $\Delta_{V_0}$ est minimal, on a $\Delta_{V_0}=\Delta_{V_0\times_{Y_{\sigma}} V}\sbt \Delta_{V}$.
Donc $y\in Y_{\sigma}(\RA_k)^{\et,\Br_{H_Y}}$, d'o\`u l'on d\'eduit (1).

 \medskip
 
 L'\'enonc\'e (2) d\'ecoule du m\^eme argument que l'\'enonc\'e (1). 
 
 Pour (3), d'apr\`es le corollaire \ref{BiDcor2.2.1}, le torseur $V\to X$ dans (\ref{BiDlem5.2e1}) est $G$-compatible.
 L'\'enonc\'e (3) d\'ecoule du m\^eme argument que l'\'enonc\'e (1).
\end{proof}

La proposition suivante g\'en\'eralise un lemme de Stoll (\cite{St}, cf. Lem. \ref{BiDlem4.2.3}).

\begin{prop}\label{BiDlem5.4}
 Soient $G$ un $k$-groupe lin\'eaire connexe et $(X,\rho)$ une $G$-vari\'et\'e lisse g\'eom\'e\-triquement int\`egre.
Supposons que $X(\RA_k)^{G-\et, \Br_G}\neq\emptyset$.
Alors, pour tout $k$-groupe fini $F$ et tout $F$-torseur $Y\to X$, 
il existe un $\sigma\in H^1(k,F)$ tel qu'il existe une composante connexe $Y'\sbt Y_{\sigma}$ qui est g\'eom\'e\-triquement int\`egre.

De plus, dans ce cas, il existe un sous $k$-groupe ferm\'e $F'\sbt F_{\sigma}$ tel que $Y'$ soit un $F'$-torseur sur $X$, 
o\`u l'action de $F'$ sur $Y'$ est induite par l'action de $F_{\sigma}$ sur $Y_{\sigma}$.
\end{prop}

\begin{proof}
Le morphisme $G\times X\xrightarrow{\rho} X$ induit un homomorphisme $\rho_{\pi_1}: \pi_1(G_{\bk})\to \pi_1(X)$.
D'apr\`es le corollaire \ref{BiDrem2.2}, l'image $\Im(\rho_{\pi_1})$ est un sous-groupe normal de $\pi_1(X)$ 
et elle est contenue dans le centre de $\pi_1(X_{\bk})$.
Pour tout $k$-groupe fini $F_1$, d'apr\`es (\ref{BiDprop2.2e}), tout $F_1$-torseur $Y_1\to X$ induit un homomorphisme $\theta_1: \pi_1(X_{\bk})\to F_1(\bk)$ \`a conjugaison pr\`es et, 
d'apr\`es la proposition \ref{BiDprop2.2} et le corollaire \ref{BiDrem2.2}, $Y_1$ est $G$-compatible si et seulement si $\theta_1\circ \rho_{\pi_1}=0$. 

D'apr\`es (\ref{BiDprop2.2e}), soit $\alpha\in H^1(\pi_1(X),F(\bk))$ un 1-cocycle qui correspond \`a $[Y]\in H^1(X,F)$.
Il existe un sous-groupe ouvert distingu\'e $\Delta \sbt \pi_1(X)$ tel que $\alpha|_{\Delta}=0$.
Soient $\Delta_{\bk}:=\Delta\cap \pi_1(X_{\bk})$ et
$\alpha_{\bk}:=\alpha|_{\pi_1(X_{\bk})}$.
Alors $\alpha_{\bk}$ est un homomorphisme $\pi_1(X_{\bk})\to F(\bk)$.

\begin{lem}
Pour trouver  $Y'$ dans la proposition \ref{BiDlem5.4}, on peut supposer que $\Delta_{\bk}\cdot \Im(\rho_{\pi_1})=\pi_1(X_{\bk})$
 et donc $\Im(\alpha_{\bk})=\Im(\alpha_{\bk}\circ \rho_{\pi_1})$.
\end{lem}

\begin{proof}
Le sous-groupe $\Im(\rho_{\pi_1})\cdot \Delta $ est ouvert normal dans $\pi_1(X)$. 
Soit $Y_2\to X$ le rev\^etement galoisien correspondant.
Par  construction, $Y_2\to X$ est un torseur $G$-compatible 
sous un $k$-groupe constant $F_2=\pi_1(X)/(\Im(\rho_{\pi_1})\cdot \Delta)$.
Par hypoth\`ese, il existe un $\sigma\in H^1(k,F_2)$ tel que $Y_{2,\sigma}(\RA_k)^{\Br_G(Y_{2,\sigma})}\neq \emptyset$.
D'apr\`es le lemme \ref{BiDlem4.2.1} et (\ref{BiDlem4.2.1-e}), il existe une composante connexe $Y_3\sbt Y_2$ telle que $Y_3$ est g\'eom\'e\-triquement int\`egre.
Ainsi $Y_3\to X$ est un torseur sous un sous-groupe ferm\'e $F_3\sbt F_{2,\sigma}$ et on a
$$\Im(\pi_1(Y_{3,\bk})\hookrightarrow \pi_1(X_{\bk}))=\pi_1(X_{\bk})\cap \Im(\pi_1(Y_2)\to \pi_1(X))=\Im(\rho_{\pi_1})\cdot \Delta_{\bk}.$$
Alors $Y_3\to X$ est $G$-compatible.
Par le lemme \ref{BiDlem5.2} (3), apr\`es avoir remplac\'e $Y_3$ par son tordu,
on peut supposer que $Y_3(\RA_k)^{G-\et,\Br_G}\neq\emptyset$. 

S'il existe un $\sigma\in H^1(k,F)$ et une composante connexe $Y'_3$ du $F_{\sigma}$-torseur $Y_{\sigma}\times_XY_3\to Y_3$ tels que $Y'_3$ soit g\'eom\'e\-triquement int\`egre, 
alors l'image de $Y'_3$ par le morphisme fini \'etale $Y_{\sigma}\times_XY_3\to Y_{\sigma} $ 
est une composante connexe $Y'$ de $Y_{\sigma}$ telle que $Y'$ soit g\'eom\'e\-triquement int\`egre.
Donc on peut remplacer $X$ par $Y_3$ et,
apr\`es avoir remplac\'e  $X$ par $Y_3$, on peut supposer que $\Delta_{\bk}\cdot \Im(\rho_{\pi_1})=\pi_1(X_{\bk})$.
Puisque  $\alpha_{\bk}(\Delta_{\bk})=0$, on a $\Im(\alpha_{\bk})=\Im(\alpha_{\bk}\circ \rho_{\pi_1})$.
\end{proof}

Dans ce cas,  puisque $\pi_1(G_{\bk})$ est commutatif, $\Im(\alpha_{\bk})$ est commutatif.
D'apr\`es le corollaire \ref{BiDrem2.2}, $\alpha_{\bk}$ induit un homomorphisme $\pi_1(X_{\bk})^{ab}\to F(\bk)$ de noyau $\Gamma_k$-invariant, car $\alpha$ est d\'efini sur $k$.
D'apr\`es le corollaire \ref{BiDcor2.1.1}, il existe un $k$-groupe fini commutatif $S$ et un $S$-torseur $\CT\to X$
 tels que $\CT$ soit g\'eom\'e\-triquement int\`egre, $S(\bk)=\Im(\alpha_{\bk})$ et que,
 dans $H^1(X_{\bk},S)\cong \Hom_{cont}(\pi_1(X_{\bk}),\Im(\alpha_{\bk}))$, on ait $[\CT_{\bk}]=\alpha_{\bk}$.

Soit $(H_Y\xrightarrow{\psi_Y}G)$ le groupe minimal compatible au $F$-torseur $Y$.
D'apr\`es la remarque \ref{BiDrem2.3}, 
$(H_Y\xrightarrow{\psi_Y}G)$ est aussi le groupe minimal compatible au  $S$-torseur $\CT$.
D'apr\`es le corollaire \ref{BiDcor2.3}, $\Ker(\psi_Y)\cong S$.
Donc $Y_4:=\CT\times_XY$ est une $H_Y$-vari\'et\'e et $Y_4\to X$ est un $(S\times F)$-torseur $H_Y$-compatible.
Donc $Y_5:=Y_4/\Ker(\psi_Y)\to X$ est un $F$-torseur $G$-compatible 
et on a un $F$-morphisme fini \'etale $\phi_5: Y_5\to Y/\Ker(\psi_Y)$.
Par hypoth\`ese, il existe un $\sigma\in H^1(k,F)$ tel que $Y_{5,\sigma}(\RA_k)^{\Br_G(Y_{5,\sigma})}\neq\emptyset $.
D'apr\`es le lemme \ref{BiDlem4.2.1}, il existe une composante connexe $Y'_5$ de $Y_{5,\sigma}$ 
telle que $Y'_5$ soit g\'eom\'e\-triquement int\`egre.
Ainsi $\phi_5(Y'_5)$ est une composante connexe de $(Y/\Ker(\psi_Y))_{\sigma}$, qui est g\'eom\'e\-triquement int\`egre.
Puisque $H_Y$ est connexe, les composantes connexes g\'eom\'e\-triques de $(Y/\Ker(\psi_Y))_{\sigma}$ et de $Y_{\sigma}$ sont les m\^emes, d'o\`u le r\'esultat.
\end{proof}

\begin{lem}\label{BiDlem5.1}
On a $X(\RA_k)^{\et, \Br_G}=X(\RA_k)^{G-\et, \Br_G}$.
\end{lem}

\begin{proof}
Il suffit de montrer que, pour tout $x\in X(\RA_k)^{G-\et, \Br_G}$, tout $k$-groupe fini $F$ 
et tout $F$-torseur $Y\xrightarrow{f}X$, 
il existe un $\sigma\in H^1(k,F)$, un $y\in Y_{\sigma}(\RA_k)^{\Br_{H_Y}(Y_{\sigma})}$ tels que $f_{\sigma}(y)=x$,
o\`u $(H_Y\xrightarrow{\psi_Y}G)$ est le groupe minimal compatible au $F$-torseur $Y$.

On fixe de tels $x,F,Y,f$.

D'apr\`es la proposition \ref{BiDlem5.4}, on peut supposer que $Y$ est g\'eom\'e\-triquement int\`egre.

D'apr\`es le corollaire \ref{BiDcor2.3}, il existe un plongement $\phi: \Ker(\psi_Y)\to F$ d'image centrale   compatible avec les actions de  $\Ker(\psi_Y)$ et de $F$ sur $Y$.
Ceci induit une suite exacte de $k$-groupes finis 
$$1\to \Ker(\psi_Y)\xrightarrow{\phi} F\xrightarrow{\phi_1} F_1\to 1 $$
qui d\'efinit $F_1$.
Alors $Y_1:=Y/\Ker(\psi_Y)\xrightarrow{f_1} X$ est un $F_1$-torseur $G$-compatible sur $X$.
De plus, $Y_1$ est lisse et g\'eom\'e\-triquement int\`egre.

Par hypoth\`ese, il existe un $\sigma_1\in H^1(k,F_1) $ et un $y_1\in Y_{1,\sigma_1}(\RA_k)^{\Br_G(Y_{1,\sigma_1})}$
tels que $f_{1,\sigma_1}(y_1)=x$.
D'apr\`es le corollaire \ref{BiDprop4.3cor}, il existe un $\sigma_0\in H^1(k,F)$ tel que $\phi_{1,*}(\sigma_0)=\sigma_1$.
Comme l'image de $\phi$ est centrale dans $F$, on a $\Ker(\psi_Y)_{\sigma_0}=\Ker(\psi_Y)$.

L'argument ci-dessus donne un $\Ker(\psi_Y)$-torseur $Y_{\sigma_0}\to Y_{1,\sigma_1}$ compatible avec l'action de $H_Y$.
D'apr\`es la proposition \ref{BiDprop4.1}, il existe un $\sigma_2\in H^1(k,\Ker(\psi_Y))$ 
 et un $y\in Y_{\sigma}(\RA_k)^{\Br_{H_Y}(Y_{\sigma})}$ avec $\sigma:=\sigma_0+\sigma_2\in H^1(k,F)$ 
 tels que $f_{\sigma}(y)=x$.
\end{proof}

\begin{lem}\label{BiDlem5.3}
On a $  X(\RA_k)^{\et, \Br}= X(\RA_k)^{\et, \Br_G}$.
\end{lem}

\begin{proof}
On peut supposer que $X(\RA_k)^{\et, \Br_G}\neq\emptyset$.
Il suffit de montrer que, pour tout $k$-groupe fini $F$ et tout $F$-torseur $f: Y\to X$, on a
$$X(\RA_k)^{\et, \Br_G}\sbt \cup_{\sigma\in H^1(k,F)}f_{\sigma}(Y_{\sigma}(\RA_k)^{\Br(Y_{\sigma})}).$$
D'apr\`es la proposition \ref{BiDlem5.4}, on peut supposer que $Y$ est g\'eom\'e\-triquement int\`egre.
L'\'enonc\'e d\'ecoule de la proposition \ref{BiDprop5.1} et du lemme \ref{BiDlem5.2} (1).
\end{proof}

\begin{proof}[D\'emonstration du th\'eor\`eme \ref{BiDthm1}.]
D'apr\`es le lemme \ref{BiDlem4.2.1} et (\ref{BiDlem4.2.1-e}), on peut supposer que $X$ est g\'eom\'e\-triquement int\`egre.
On  obtient le th\'eor\`eme 
par combinaison du lemme \ref{BiDlem5.3} et du lemme \ref{BiDlem5.1}.
\end{proof}

\begin{rem}
Rappelons les cat\'egories $\mathbf{AB}$ et $\mathbf{GX}$ dans \S \ref{BiD4}.
On fixe un objet $(G,X)\in \mathbf{GX}$. 

Soit $\mathbf{GX}_X$ l'ensemble des objets $(H,Y)\in \mathbf{GX}$ tels qu'il existe un morphisme $(\psi,f): (H,Y)\to (G,X)$ dans $\mathbf{GX}$ avec $\psi, f$ finis.

Dans toute cette section (\S \ref{BiD6}), pour \'etablir le th\'eor\`eme \ref{BiDthm1} de $(G,X)$,
 l'hypoth\`ese que $G$ est lin\'eaire et la notion de sous-groupe de Brauer invariant sont utilis\'es seulement 
 pour appliquer la proposition \ref{BiDcor2.3.1}, la proposition \ref{BiDprop4.1}, 
le corollaire \ref{BiD4cor1}, le lemme \ref{BiDlem4.2.1} et le corollaire \ref{BiDprop4.3cor} 
\`a l'\'el\'ement dans $\mathbf{GX}_X$.
Donc, cette section a essentiellement montr\'e :

{\it 
pour tout foncteur contravariant $B(-,-): \mathbf{GX}\to \mathbf{AB}$ qui associe au couple $(H,Y)$ un sous-groupe $B(H,Y)\sbt \Br(Y)$,
si l'on peut \'etablir  la proposition \ref{BiDcor2.3.1}, la proposition \ref{BiDprop4.1}, le corollaire \ref{BiD4cor1}, le lemme \ref{BiDlem4.2.1} et le corollaire \ref{BiDprop4.3cor}
pour tout \'el\'ement dans $\mathbf{GX}_X$ (en rempla\c{c}ant tout groupe de Brauer invariant par le $B(-,-)$ correspondant), 
alors on a $X(\RA_k)^{\et,\Br}=X(\RA_k)^{G-\et,B(G,-)}$, o\`u $X(\RA_k)^{G-\et,B(G,-)}$ est d\'efini de la m\^eme fa\c{c}on que $X(\RA_k)^{G-\et,\Br_G}$.
}

\end{rem}

\section{D\'emonstration des th\'eor\`emes \ref{BiDthm2} et \ref{BiDcor2}}

Dans toute cette section,  $k$ est un corps de nombres. 
Sauf  mention explicite du contraire, une vari\'et\'e est  une $k$-vari\'et\'e.

\begin{proof}[D\'emonstration du th\'eor\`eme \ref{BiDthm2}]
D'apr\`es le lemme \ref{BiDlem4.2.1} et (\ref{BiDlem4.2.1-e}), on peut supposer que $Z$ est g\'eom\'e\-triquement int\`egre.
Si $G$ est connexe, l'\'enonc\'e d\'ecoule du th\'eor\`eme \ref{BiDthm1} et de la proposition \ref{BiDprop4.2}.
Si $G$ est fini, d'apr\`es la proposition \ref{BiDlem5.4}, on peut supposer que $X$ est g\'eom\'e\-triquement int\`egre, et le r\'esultat d\'ecoule du lemme \ref{BiDlem5.2} (2).

Pour \'etablir le cas g\'en\'eral, on reprend certains arguments de \cite{D09} et \cite{CDX}.
Il existe une suite exacte
$$1\to N\to G\xrightarrow{\psi} F\to 1$$
de $k$-groupes lin\'eaires avec $N$ un $k$-groupe lin\'eaire connexe et $F$ un $k$-groupe fini.
Alors 
$$h: U:=X/N\to Z$$
 est un $F$-torseur. Notons $q: X\to U$.
Pour un $z\in Z(\RA_k)^{\et,\Br}$, il existe un $\sigma\in H^1(k,F)$ et un $u\in U_{\sigma}(\RA_k)^{\et,\Br}$ tels que $h_{\sigma}(u)=z$.
D'apr\`es la proposition \ref{BiDprop4.3} (1), il existe un $\alpha_0\in H^1(k,G)$ tel que $\psi_*(\alpha_0)=\sigma$.
Ceci induit une suite exacte 
$$1\to N'\xrightarrow{\phi} G_{\alpha_0}\xrightarrow{\psi_{\alpha_0}} F_{\sigma}\to 1 $$
de $k$-groupes lin\'eaires.
Alors $N'_{\bk}\cong N_{\bk}$ et $N'$ est un $k$-groupe lin\'eaire connexe. Ainsi $q_{\alpha_0}: X_{\alpha_0}\to U_{\sigma}$ est un $N'$-torseur.
Donc il existe un $\beta\in H^1(k,N')$ et un $x\in (X_{\alpha_0})_{\beta}(\RA_k)^{\et,\Br}$ tels que $(q_{\alpha_0})_{\beta}(x)=u$.
Soit $\alpha:=\alpha_0+\phi_*(\beta)$. 
Alors $(X_{\alpha_0})_{\beta}=X_{\alpha}$ et $p_{\alpha}=h_{\sigma}\circ (q_{\alpha_0})_{\beta}$.
Donc $x\in X_{\alpha}(\RA_k)^{\et,\Br}$ et $p_{\alpha}(x)=z$, d'o\`u le r\'esultat.
\end{proof}

\begin{proof}[D\'emonstration du th\'eor\`eme \ref{BiDcor2}]
Ceci d\'ecoule du th\'eor\`eme \ref{BiDthm2} et de \cite[Thm. 1.5]{CDX}.
\end{proof}

\begin{proof}[D\'emonstration du corollaire \ref{BiDcor1}]
 Pour tout $k$-groupe fini $F$ et tout $F$-torseur $G$-compatible $f: Y\to X$, d'apr\`es le corollaire \ref{BiDrem2.2} (4), il existe un $F$-torseur $M$ sur $k$ tel que $Y\cong M\times_k X$ comme $F$-torseurs.
 Alors il existe un $\sigma_0\in H^1(k,F)$ tel que $Y_{\sigma_0}\cong F\times X$.
 Donc 
 $$X(\RA_k)^{\Br_G(X)}\sbt f_{\sigma_0}(Y_{\sigma_0}(\RA_k)^{\Br_G(Y_{\sigma_0})})\sbt \cup_{\sigma\in H^1(k,F)} f_{\sigma}(Y_{\sigma}(\RA_k)^{\Br_G(Y_{\sigma})}). $$
 Ainsi $ X(\RA_k)^{G-\et, \Br_G}=X(\RA_k)^{\Br_G(X)}$ et le r\'esultat d\'ecoule du th\'eor\`eme \ref{BiDthm1}.
\end{proof}

\bigskip

\noindent{\bf Remerciements.}
Je remercie tr\`es chaleureusement Jean-Louis Colliot-Th\'el\`ene et Cyril Demarche pour leurs commentaires.

\bibliographystyle{alpha}

\begin{thebibliography}{Gro}
\bibitem[Ba]{Bale} F. Balestrieri: \emph{Iterating the algebraic etale-Brauer set}, Journal of Number Theory 182 (2018), 284--295.
\bibitem[BD]{BD} M. Borovoi et C. Demarche: \emph{Manin obstruction to strong approximation for homogeneous spaces}, Comment. Math. Hev. 88 (2013), 1-54.
\bibitem[BS]{BS} M. Brion et T. Szamuely: \emph{Prime-to-p \'etale covers of algebraic groups and homogeneous spaces}, Bull. Lond. Math. Soc. 45 (2013), no. 3, 602--612. 
\bibitem[Cod]{Co} B. Conrad:  \emph{Weil and Grothendieck approaches to adelic points}, Enseign. Math. 58 (2012), 61--97.
\bibitem[CDX]{CDX} Y. Cao, C. Demarche, F. Xu: \emph{Comparing descent obstruction and Brauer-Manin obstruction for open varieties}, Trans. Amer. Math. Soc. 371 (2019), no. 12, 8625--8650.
\bibitem[C18]{C1} Y. Cao: \emph{Approximation forte pour les vari\'et\'es avec une action d'un groupe lin\'eaire}, Compositio math. 154 (2018), 773--819. 
\bibitem[CLX]{CLX} Y. Cao, Y. Liang, F. Xu: \emph{Arithmetic purity of strong approximation},  Journal de Math\'ematiques Pures et Appliqu\'ees 132 (2019), 334--368.
\bibitem[CT08]{CT07} J.-L. Colliot-Th\'el\`ene: \emph{R\'esolutions flasques des groupes lin\'eaires connexes}, J. reine angew. Math. 618 (2008), 77-133.
\bibitem[CTSa]{CTS1} J.-L. Colliot-Th\'el\`ene et J.-J. Sansuc:  \emph{Principal homogeneous spaces under flasque tori, applications}, Journal of Algebra 106 (1987) 148-205.
 \bibitem[CTSb]{CTS} J.-L. Colliot-Th\'el\`ene et J.-J. Sansuc:  \emph{La descente sur les vari\'et\'es rationnelles, II}, Duke Math. J. 54 (1987) 375-492.
\bibitem[D09a]{D09} C. Demarche: \emph{Obstruction de descente et obstruction de Brauer-Manin \'etale}, Algebra \& Number Theory 3 (2009) 237--254.
\bibitem[D09b]{Dth} C. Demarche: \emph{M\'ethodes cohomologiques pour l'\'etude des points rationnels sur les espaces homog\`enes},  Th\`ese de doctorat, Universit\'e Paris-Sud (2009).
\bibitem[Fu]{Fu} L. Fu: \emph{\'Etale cohomology theory},  Nankai Tracts in Mathematics, Vol. 13, World Scientific, 2011.
\bibitem[HS02]{HS02} D. Harari, A. N. Skorobogatov:  \emph{Non-abelian cohomology and rational points}, Compositio Math. 130 (2002) 241--273.
\bibitem[HS13]{HS13} D. Harari, A. N. Skorobogatov:  \emph{Descent theory for open varieties}, Torsors, \'etale homotopy and applications to rational points. LMS Lecture Note Series 405, Cambridge University Press (2013), 250--279.
\bibitem[KS]{KS} M. Kashiwara et P. Schapira, \emph{Categories and sheaves}, Grundlehren der Mathematischen Wissenschaften, vol. 332, Springer-Verlag, Berlin, 2006.
\bibitem[LX]{LX} Q. Liu et F. Xu \emph{Very strong approximation for certain algebraic varieties}, Math. Ann. 363 (2015) 701--731.
\bibitem[Ma]{Ma} Y. I. Manin: \emph{Le groupe de Brauer-Grothendieck en g\'eom\'etrie diophantienne}, Actes du Congr\`es International des Math\'ematiciens (Nice, 1970), Tome 1, pp. 401--411. Gauthier-Villars, Paris, 1971. 
\bibitem[Mi]{Mi80} J. S. Milne: \emph{\'Etale Cohomology}, Princeton Math. Ser 33, Princeton University Press, Princeton 1980.
\bibitem[Miy]{Miy} M. Miyanishi: \emph{On the algebraic fundamental group of an algebraic group}, J. Math. Kyoto Univ. 12-2 (1972), 351--367.
\bibitem[Po10]{P} B. Poonen: \emph{Insufficiency of the Brauer-Manin obstruction applied to \'etale covers}, Ann. of Math. 171 (2010) 2157--2169.
\bibitem[Po]{P1} B. Poonen: \emph{Rational points on varieties},  Graduate Studies in Mathematics, 186, American Mathematical Society, Providence, RI, 2017.
\bibitem[S]{S} J.-J. Sansuc: \emph{Groupe de Brauer et arithm\'etique des groupes alg\'ebriques lin\'eaires sur un corps de nombres}, J. reine angew   Math. 327 (1981), 12-80.
\bibitem[Se65]{Se} J.-P. Serre: \emph{Cohomologie Galoisienne}, Lecture Notes in Mathematics, vol 5, Springer, Berlin, 1965.
\bibitem[SGA1]{SGA1}  A. Grothendieck et al.: \emph{Rev\^etements \'etales et groupe fondamental, (SGA 1) }, Lecture notes in mathematics 224, Berlin; New York, Springer--Verlag 1971.
\bibitem[SGA4]{SGA4} A. Grothendieck, M. Artin, L. Verdier et al.: \emph{Th\'eorie des topos et cohomologie \'etale des sch\'emas, (SGA 4)}, Lecture Notes in Math. 269, 270, 305, Springer-Verlag (1972-1973).
\bibitem[SGA4$\frac{1}{2}$]{SGA4.5} P. Deligne et al.: \emph{Cohomologie \'etale (SGA $4\frac{1}{2}$)}, Lecture Notes in Math. 569, Springer-Verlag (1977).
\bibitem[Sk99]{Sk99} A. N. Skorobogatov: \emph{Beyond the Manin obstruction}, Invent. Math. 135 (1999), no. 2, 399--424.
\bibitem [Sk01]{sko} A. N. Skorobogatov: \emph{ Torsors and Rational Points}, Cambridge Tracts in Mathematics, vol. 144,    Cambridge University Press, 2001.
\bibitem[Sk09]{Sk1} A. N. Skorobogatov: \emph{Descent obstruction is equivalent to \'etale Brauer-Manin obstruction}, Math. Ann. 344 (2009) 501--510.
\bibitem[St]{St} M. Stoll: \emph{Finite descent obstructions and rational points on curves}, Algebra Number Theory 1 (2007) 349--391.
\bibitem [SZ]{SZ} A. N. Skorobogatov, Y. G. Zarhin: \emph{The Brauer group and the Brauer-Manin set of products of varieties},
J. Eur. Math. Soc. 16 (2014) 749--769.
\bibitem[Sz]{Sz} T. Szamuely: \emph{Galois Groups and Fundamental Groups}, Cambridge Studies in Advanced Mathematics 117, Cambridge University Press, Cambridge, 2009.
\end{thebibliography}
\end{document}